\documentclass[a4paper]{amsart}
\usepackage[latin1]{inputenc}
\usepackage{amsmath,amsthm,amsfonts,amssymb,amscd,amsbsy,dsfont,hyperref,graphicx}
\usepackage[all,knot]{xy}

\hypersetup{
    unicode=false,
    pdftoolbar=true,
    pdfmenubar=true,
    pdffitwindow=false,
    pdfstartview={FitH},
    pdftitle={Genericity of nondegenerate geodesics with general boundary conditions},
    pdfauthor={Bettiol, R. and Giamb\`o, R.},
	pdfsubject={57R45, 57R70, 57N75, 58E10},   
    pdfkeywords={}, 
    pdfnewwindow=true,
    colorlinks=true,  
    linkcolor=blue,
    citecolor=red,
}

\newtheorem{theorem}{Theorem}[section]
\newtheorem{lemma}[theorem]{Lemma}
\newtheorem{proposition}[theorem]{Proposition}

\theoremstyle{definition}
\newtheorem{definition}[theorem]{Definition}

\theoremstyle{remark}
\newtheorem{remark}[theorem]{Remark}
\newtheorem{example}[theorem]{Example}

\numberwithin{equation}{section}

\newcommand{\ev}{\operatorname{ev}}
\newcommand{\R}{\mathds{R}}
\newcommand{\N}{\mathds{N}}
\newcommand{\p}{\mathcal{P}}
\newcommand{\op}{\Omega_{\p}}
\newcommand{\met}{\ensuremath\operatorname{Met}}
\newcommand{\lin}{\ensuremath\operatorname{Lin}}
\newcommand{\gr}{\ensuremath\operatorname{Gr}}
\newcommand{\chr}{\ensuremath\Gamma}
\newcommand{\sect}{{\boldsymbol{\Gamma}}}
\newcommand{\D}{\boldsymbol{\operatorname{D}}}
\newcommand{\s}{\ensuremath{\mathcal{S}}}

\title[Genericity under general boundary conditions]{Genericity of nondegenerate geodesics with general boundary conditions}
\author[R. G. Bettiol]{Renato G. Bettiol}
\address{Departamento de Matem\'atica \hfill\break\indent Universidade de S\~ao Paulo \hfill\break\indent
Rua do Mat\~ao, 1010 \hfill\break\indent 05508-090 S\~ao Paulo, SP, Brazil}
\email{rbettiol@ime.usp.br; renatobettiol@gmail.com}
\author[R. Giamb\`o]{Roberto Giamb\`o}
\address{Dipartimento di Matematica e Informatica \hfill\break\indent
Universit\`a di Camerino\hfill\break\indent 62032 Camerino, Italy}
\email{roberto.giambo@unicam.it}
\date{October 21st, 2009. \emph{Subject Classification} (MSC2010): \textrm{57R45, 57R70, 57N75, 58E10}}
\thanks{The first author was supported by Fapesp, Brazil, Grant 2008/07604-0.}

\begin{document}

\begin{abstract}
Let $M$ be a possibly noncompact manifold. We prove, generically in the $C^k$--topology ($2\leq k\leq +\infty$), that semi--Riemannian metrics of a given index on $M$ do not possess any degenerate geodesics satisfying suitable boun\-da\-ry conditions. This extends a result of Biliotti, Javaloyes and Piccione \cite{biljavapic} for geodesics with fixed endpoints to the case where endpoints lie on a compact submanifold $\p\subset M\times M$ that satisfies an admissibility condition. Such condition holds, for example, when $\p$ is transversal to the diagonal $\Delta\subset M\times M$. Further aspects of these boundary conditions are discussed and general conditions under which metrics without degenerate geodesics are $C^k$--generic are given.
\end{abstract}

\maketitle
\tableofcontents

\section{Introduction}

Genericity of properties of flows is a widely explored topic in dynamical systems, particularly regarding geodesic flows. A well known example is the so--called \emph{bumpy} metric theorem (first stated by Abraham \cite{abraham},
with complete proof by Anosov \cite{anosov}). This asserts that Riemannian metrics on a compact manifold $M$ without degenerate periodic geodesics are generic relatively to the $C^k$--topology ($2\le k\le+\infty$).

Counterexamples by Meyer and Palmore \cite{mp} point out that abstract Hamiltonian systems cannot be considered for generalizations of the bumpy theorem to a more comprehensive class of dynamical flows. Basically, the dynamics of solutions differ in distinct energy levels, and hence the nondegeneracy property fails to be generic. On the other hand, it is interesting to extend results on geodesic flows to a more general \emph{semi--Riemannian} setting. Motivation for studying generic properties of semi--Riemannian geodesic flows also comes from Morse theory. Indeed, a crucial assumption for developing a Morse theory for geodesics between fixed points is that the two arbitrarily fixed distinct points must be nonconjugate. Recent works by Abbondandolo and Majer \cite{AbbMej2,AbbMej} connect Morse relations for critical points of the semi--Riemannian energy functional to the homology of a doubly infinite chain complex, the Morse--Witten complex. They also prove  stability of this homology with respect to small perturbations of the metric structure. Thus, it is important to ask whether it is possible to perturb a metric in such a way that the nonconjugacy property between two points is preserved. A positive answer to this question is given by a recent work of Biliotti, Javaloyes and Piccione \cite{biljavapic}, which proves genericity of semi--Riemannian metrics on a (possibly noncompact) manifold $M$ without degenerate geodesics joining two arbitrarily fixed distinct points $p,q\in M$.

The goal of this paper is to extend this result when more general boundary conditions on geodesics are considered. Our main result asserts that such non\-de\-ge\-ne\-racy property is also generic considering geodesics with endpoints in an \emph{admissible general boundary condition}. More precisely, consider $(M,g)$ a $n$--dimensional semi--Riemannian manifold of index $\nu$. A \emph{general boundary condition} for the geo\-de\-sic variational problem on $M$ is an arbitrary compact submanifold $\p$ of the product $M\times M$ that does not have a particular $\nu$--topological obstruction\footnote{This obstruction is explained in detail in Remark \ref{semiriemanniansubmanifold}. Our assumption on the submanifold $\p\subset M\times M$ is that it admits semi-Riemannian metrics of index $n$ that are given as restrictions of product metrics $g\oplus (-g)$ on $M\times M$, where $g$ is a metric of index $\nu$ on $M$.}. Geodesics considered are affinely parametrized $g$--geodesics whose endpoints lie in $\p$ and whose tangent vectors are orthogonal to $\p$ at these points. Such geodesics will be called \emph{$(g,\p)$--geodesics}. We find suitable admissibility conditions on $\p$ (see Definition \ref{def:admgbc}) under which the set of metrics $g$ of index $\nu$ such that all $(g,\p)$--geodesics are nondegenerate is $C^k$--generic in some appropriate space of semi--Riemannian metric structures on $M$. This is the content of our main result, Theorem \ref{bigone}.

The case studied in \cite{biljavapic} corresponds to $\p=\{p\}\times\{q\}$, with the hypothesis that $p\ne q$. Therefore, the case of nonconstant geodesic loops at a point $p$ is left open, and it is conjectured that the same genericity statement holds. Theorem \ref{bigone} answers positively this conjecture, once $\p=\{p\}\times\{p\}$ satisfies the mentioned admissibility conditions (Definition \ref{def:admgbc}). Such conditions hold, for instance, when $\p$ does not intersect the diagonal $\Delta\subset M\times M$, or, more generally, when it intersects $\Delta$ transversally (Proposition \ref{admissibility}).

Nevertheless, these admissibility conditions mentioned trivially fail for \emph{bumpy} boundary conditions $\p=\Delta$. In this particular case, that corresponds to periodic geodesics, a similar nondegeneracy genericity statement holds due to the recent proof of the semi--Riemannian version of the bumpy metric theorem. This is a result of Biliotti, Javaloyes and Piccione \cite{biljavapic2}, using equivariant variational techniques, rather than dynamical. Such result is used in a nontrivial way in the proof of Theorem \ref{bigone} if $\p\cap\Delta\neq\varnothing$. In addition, it is important to stress that the genericity results of \cite{biljavapic,biljavapic2} combined do \emph{not} automatically imply genericity of metrics without de\-ge\-ne\-ra\-te geodesics under general boundary conditions. Essentially, the de\-ge\-ne\-racy notions considered are different (see Remark \ref{degeneracynotions}). Suppose $\p\cap\Delta\neq\varnothing$. Then there may be nontrivial Jacobi fields that degenerate a periodic geodesic as a periodic geodesic, but not as a $(g,\p)$--geodesic. Therefore, a more involved argument is required. In general lines, this is done using the semi--Riemannian bumpy metric theorem to ensure that one may first restrict to metrics without degenerate periodic geodesics, which are generic. Only then an abstract criterion (Proposition \ref{abstractgenericity}) is used to prove genericity of metrics without degenerate geodesics with boundary conditions $\p$. For this, a particularly degenerate class of geodesics is studied (Subsection \ref{sec:stdeg}) and the admissibility property is used in a crucial form.

Several geometric interpretations of this result are possible. For instance, consider $P\subset M$ a fixed compact submanifold without $\nu$-topological obstructions (Remark \ref{semiriemanniansubmanifold}) and $q\in M$ is a fixed point. Since $\p=P\times\{q\}$ satisfies the admissibility conditions mentioned above (Example \ref{adexgbc}), our result can be applied. In this setting, it asserts that $q$ is not focal to $P$ in a $C^k$--generic semi-Riemannian metric. It extends the genericity of the nonconjugacy property for two fixed distinct points, that correspond to the fixed endpoints case treated in \cite{biljavapic}.

We now provide a short overview of the paper topics. A few preliminaries and notation are established in Section \ref{sec:first}. We recall the definition of \emph{$C^k$--Whitney type Banach space of tensor fields} over a manifold and explore some elementary aspects of semi--Riemannian geodesics. In Section \ref{abstractresult}, we reproduce an abstract ge\-ne\-ri\-ci\-ty criterion (Proposition \ref{abstractgenericity}) used in the proof of several genericity results. This theorem is present in both \cite{biljavapic} and \cite{chillingsworth} and was successfully used to establish the genericity results of \cite{biljavapic} mentioned before. It follows the lines of a standard transversality argument by White \cite{white}, that uses the Sard--Smale theorem \cite{smale} for a family of nonlinear Fredholm functionals $f_x$ on a Hilbert manifold, parametrized in a Banach manifold. Briefly, it asserts that the values of $x$ such that $f_x$ has only nondegenerate critical points is generic, under suitable regularity conditions. This abstract genericity criterion is also used in the proof of the semi--Riemannian bumpy metric theorem \cite{biljavapic2}, and in other contexts such as \cite{ggp}. In Section \ref{sec:gbc}, we introduce the concept of admissible general boundary conditions, and explore a few particular cases. Furthermore, the admissibility of a large class of boundary conditions is established in Proposition \ref{admissibility}. In Section \ref{secgenericity}, we prove our main result, Theorem \ref{bigone}. Finally, in Subsection \ref{sec:cinfty} it is improved to the weak $C^\infty$--topology, although in principle the arguments used in the proof do not apply directly, due to lack of regularity of the metric tensors space.

\section{Preliminaries and notations}\label{sec:first}

Throughout the text $M$ will denote a smooth manifold of finite dimension $n$, and by \emph{smooth} we will always mean of class $C^\infty$. Regarding differentiability of tensors, particularly metric tensors, which will usually be of class $C^k$, we will implicitly consider $k\geq 2$. Furthermore, $g_\mathrm R$ will denote a fixed complete Riemannian metric on $M$.

\subsection{Banach spaces of sections}

Let $p:E\rightarrow M$ a vector bundle. Then $\sect^k(E)$ is the space of $C^k$ sections of $E$, and in the case $E=TM^*\otimes TM^*$, we denote by $\sect^k_{\mbox{\tiny sym}}(TM^*\otimes TM^*)$ the set of $C^k$ sections $s$ such that $s_x:T_xM\times T_xM\rightarrow\R$ is symmetric for all $x$. Given another smooth manifold $N$ and a smooth map $f:N\rightarrow M$, the pull--back by $f$ of vector bundle $E$ will be denoted $f^*E$. Finally, $\met_\nu^k(M)$ is the set of all semi--Riemannian $C^k$ metric tensors of index $\nu\in\{0,\dots,n\}$, which is a subset of $\sect^k_{\mbox{\tiny sym}}(TM^*\otimes TM^*)$.

If $M$ is compact, $\sect^k_{\mbox{\tiny sym}}(TM^*\otimes TM^*)$ has a natural Banach space structure, and $\met_\nu^k(M)$ is an open subset. Adopting the approach in \cite[Subsection 4.1]{biljavapic} to endow the space of tensors over a noncompact manifold $M$ with a Banach space structure, consider the following.

\begin{definition}
A vector subspace $\mathcal{E}$ of $\sect^k_{\mbox{\tiny sym}}(TM^*\otimes TM^*)$ is called a {\it $C^k$--Whitney type Banach space of tensor fields} over $M$ if

\begin{itemize}
\item[(i)] $\mathcal{E}$ contains all tensor fields in $\sect^k_{\mbox{\tiny sym}}(TM^*\otimes TM^*)$ having compact support;
\item[(ii)] $\mathcal{E}$ is endowed with a Banach space norm $\|\cdot\|_{\mathcal{E}}$ with the property that $\|\cdot\|_{\mathcal{E}}$--convergence of a sequence implies convergence in the weak Whitney $C^k$--topology.
\end{itemize}

The second condition means that given any sequence $\{\mathfrak{b}_\alpha\}$ and an element $\mathfrak{b}_\infty\in\mathcal{E}$ such that $\lim \|\mathfrak{b}_\alpha-\mathfrak{b}_\infty\|_\mathcal{E}=0$, for each compact set $K\subset M$, the restriction $\mathfrak{b}_\alpha|_K$ tends to $\mathfrak{b}_\infty|_K$ in the $C^k$ topology as $\alpha$ tends to $\infty$.
\end{definition}

\begin{remark}
Using the auxiliary Riemannian metric $g_\mathrm{R}$ on $M$ it is possible to construct $C^k$--Whitney type Banach space of tensors on $M$ as follows. Firstly, we observe that the Levi--Civita connection $\nabla^\mathrm{R}$ of $g_\mathrm{R}$ induces a connection on all vector bundles over $M$ obtained with functorial constructions from the tangent bundle $TM$. Furthermore, for each $r,s\in\N$, $g_\mathrm{R}$ induces canonical Hilbert space norms on each tensor bundle $T_xM^{*(r)}\otimes T_xM^{(s)}$, which will be denoted $\|\cdot\|_\mathrm{R}$. Finally, we define $\sect^k_{\mbox{\tiny sym}}(TM^*\otimes TM^*;g_\mathrm{R})$ as the subset of $\sect^k_{\mbox{\tiny sym}}(TM^*\otimes TM^*)$ consisting of all sections $\mathfrak{b}$ such that $$\|\mathfrak{b}\|_k=\max_{0\leq i\leq k}\Big[\sup_{x\in M} \Big\|(\nabla^\mathrm{R})^i \mathfrak{b}(x)\Big\|_\mathrm{R}\Big]<\infty.$$

When $M$ is compact, $\sect^k_{\mbox{\tiny sym}}(TM^*\otimes TM^*;g_\mathrm{R})=\sect^k_{\mbox{\tiny sym}}(TM^*\otimes TM^*)$. The norm $\|\cdot\|_k$ defined above turns $\sect^k_{\mbox{\tiny sym}}(TM^*\otimes TM^*;g_\mathrm{R})$ into a separable normed space, which is complete if the Riemannian metric $g_\mathrm{R}$ is complete. It is then easy to see that $\sect^k_{\mbox{\tiny sym}}(TM^*\otimes TM^*;g_\mathrm{R})$ is a $C^k$--Whitney type Banach space of tensors.
\end{remark}

We will use the following result proved in \cite[Lemma 2.4]{biljavapic}, concerning the existence of global section of a vector bundle with prescribed value and covariant derivative along a sufficiently small curve.

\begin{lemma}\label{extension}
Let $p:E\rightarrow B$ be a smooth vector bundle endowed with a connection $\nabla$, $\gamma:[a,b]\rightarrow M$ a smooth curve and $v\in\sect(\gamma^*TM)$ a smooth vector field along $\gamma$, such that $v(t_0)$ is not parallel to $\dot{\gamma}(t_0)$ for some $t_0\in \; ]a,b[$. Then there exists an open interval $I\subset [a,b]$ containing $t_0$ with the property that, given smooth sections $H$ and $K$ of $\gamma^*E$ with compact support in $I$ and given any open set $U$ containing $\gamma(I)$, there exists $h\in\sect(E)$ with compact support contained in $U$, such that $h_{\gamma(t)}=H_t$ and $\nabla_{v(t)} h=K_t$ for all $t\in I$.
\end{lemma}

\subsection{Semi--Riemannian basics}

We now recall some elementary concepts of semi--Riemannian geometry and make a few conventions. Given any symmetric $(0,2)$--tensor $\mathfrak{b}$ on $M$, for instance a semi--Riemannian metric, for all $x\in M$, the bilinear map $\mathfrak b(x)$ will be identified with the linear operator $$\mathfrak{b}(x):T_xM\longrightarrow T_xM^*.$$

Let $\nabla$ be an arbitrary symmetric connection on $TM$. Given another connection $\nabla'$, the difference $$\chr=\nabla'-\nabla$$ is a $(1,2)$--tensor called the {\it Christoffel tensor} of $\nabla'$ relatively to $\nabla$, which can be computed using Koszul's formula. The connection $\nabla$ induces a covariant derivative of vector fields along curves on $M$, which will be denoted $\D$. In case $\nabla^g$ is the Levi--Civita connection of $g\in\met_\nu^k(M)$, the corresponding operator of covariant derivative for vector fields along curves will be denoted $\D^g$; and for the fixed Riemannian metric $g_\mathrm R$, it will be simply denoted $\D^\mathrm R$. The Riemannian length of a curve $\gamma:[a,b]\rightarrow M$ with respect to $g_\mathrm R$ will be denoted $$L_\mathrm R(\gamma)=\int_a^b \left\| \dot\gamma(t) \right\|_\mathrm R \;\mathrm{d}t.$$

The sign convention adopted for the curvature tensor of $\nabla^g$ is $$R^g(X,Y)=[\nabla^g_X,\nabla^g_Y]-\nabla^g_{[X,Y]}.$$ Let $\gamma$ be a solution of the $g$--geodesic equation on $M$. Then $\gamma$ will be called a \emph{$g$--geodesic} only if it is \emph{affinely parametrized}. A {\it Jacobi field} along a $g$--geodesic $\gamma$ is a smooth section $J\in\sect^k(\gamma^*TM)$ satisfying the {\it Jacobi equation} $$(\D^g)^2 J=R^g(\dot{\gamma},J)\dot{\gamma}.$$ The endpoints of $\gamma$ are said to be \emph{conjugate} along $\gamma$ if there exists a nontrivial Jacobi field along $\gamma$ that vanishes at both endpoints of $\gamma$. Affine multiples of the tangent field $\dot\gamma$ are trivially Jacobi fields, and conversely, the only Jacobi fields along $\gamma$ that are everywhere parallel to $\dot\gamma$ are its affine multiples. In addition, Jacobi fields are only parallel to $\dot\gamma$ at isolated points.

\begin{lemma}\label{parallelfinite}
Let $\gamma:[a,b]\rightarrow M$ be a $g$--geodesic and $J$ a nontrivial Jacobi field along $\gamma$, that is not everywhere parallel to $\dot{\gamma}$. Then $\mathcal{D}=\{t\in[a,b]:J(t) \mbox{ is parallel to }\dot{\gamma}\}$ consists only of isolated points, hence is finite.
\end{lemma}

\begin{proof}
Consider a basis of $T_{\gamma(a)}M$ given by $(\dot{\gamma}(a),e_2,\dots,e_n)$ and its parallel transport along $\gamma$ creating a frame $(e_1(t),e_2(t),\dots,e_n(t))$, with $e_1(t)=\dot{\gamma}(t)$. Then, writing $J=\sum_{i=1}^n J_i(t)e_i(t)$, $J$ is parallel to $\dot{\gamma}$ at time $t$ if and only if $J_i(t)=0$, for $i\geq 2$. Suppose that there exists a limit $t_\infty\in [a,b]$ of a sequence $\{t_\alpha\}$ of different elements of $\mathcal{D}$. From continuity of $J$ it follows that $t_\infty \in\mathcal{D}$. Thus for each $i\geq 2$, the coordinate function $J_i(t)$ has a convergent sequence of zeros $\{t_\alpha\}$ and hence $J_i'(t_\infty)=0$. Therefore, the covariant derivative $\D^g J(t_\infty)$ is also parallel to $\dot{\gamma}$.

It is then possible to find $c_1,c_2\in\R$ such that $\tilde{J}=(c_1+c_2t)\dot{\gamma}(t)$ satisfies $\tilde{J}(t_\infty)=J(t_\infty)$ and $\D^g \tilde{J}(t_\infty)=\D^g J(t_\infty)$. Since the Jacobi equation is a second order linear ODE, $\tilde{J}=J$. Hence $J$ is always parallel to $\dot{\gamma}$, a contradiction.
\end{proof}

\subsection{Geodesics self intersections}

The following elementary results will be used later to deal with geodesic self intersection problems.

\begin{lemma}\label{finiteintersection}
Let $\gamma_i:[a_i,b_i]\rightarrow M$ two $g$--geodesics. Then the set of points where these geodesics intersect is finite, unless one is an affine reparametrization of the other.
\end{lemma}

\begin{proof}
Since the images of $\gamma_1$ and $\gamma_2$ are compact, if there were infinitely many intersection points, there would be an accumulation intersection point $p=\gamma_1(t)=\gamma_2(s)$. Consider $U$ a normal neighborhood of $p$. If $\dot{\gamma_1}(t)$ and $\dot{\gamma_2}(s)$ are linearly independent, since there are infinitely many points near $p$ such that $\gamma_1$ and $\gamma_2$ coincide in $U$, there is an obvious contradiction to injectivity of the exponential map on $U$. Otherwise, if $\dot{\gamma_1}(t)$ and $\dot{\gamma_2}(s)$ are linearly dependent, then $\gamma_1$ and $\gamma_2$ are affine reparametrizations of each other.
\end{proof}

\begin{proposition}\label{selfintersections}
Let $\gamma:[0,1]\rightarrow M$ be a $g$--geodesic in $M$. If the set
\[\mathcal{I}=\big\{(t,s)\in [0,1]\times [0,1] : t\neq s, \gamma(t)=\gamma(s)\big\}\]
is infinite, then $\gamma$ is periodic with period $\omega<1$.
\end{proposition}

\begin{proof}
If $\mathcal{I}$ is infinite, there exists an accumulation point $\left(\overline{t},\overline{s}\right)\in\mathcal{I}$. The local injectivity of $\gamma$ implies that $\overline{t}\neq\overline{s}$, suppose $\overline{t}<\overline{s}$. Take $\varepsilon>0$ small, and define $\gamma_1=\gamma_{\varepsilon}\big|_{\left[\overline{t}-\varepsilon,\overline{t}+\varepsilon\right]}$ and $\gamma_2=\gamma_{\varepsilon}\big|_{\left[\overline{s}-\varepsilon,\overline{s}+\varepsilon\right]}$, where $\gamma_\varepsilon$ is the extension of $\gamma$ to $[-\varepsilon,1+\varepsilon]$. Since $\gamma_1$ and $\gamma_2$ are defined on compact intervals and intersect infinitely many times, from Lemma \ref{finiteintersection}, one is an affine reparametrization of the other. Moreover, both are restrictions of the same geodesic $\gamma_\varepsilon$, hence $\gamma_1(t+\omega)=\gamma_2(t)$ for $t\in \left[\overline{t}-\varepsilon,\overline{t}+\varepsilon\right]$, where $\omega=\overline{s}-\overline{t}\leq 1$. Therefore $\dot{\gamma_1}(\overline{t})=\dot{\gamma_2}(\overline{s})$, hence $\gamma$ is periodic with period $\omega\leq 1$. If $\overline{t}=0$ and $\overline{s}=1$, one can easily derive a contradiction with local injectivity of $\gamma$ around $0$, which implies $\omega <1$.
\end{proof}

\subsection{Submanifold geometry}

We end this section recalling some classic facts about submanifolds of a semi--Riemannian manifold $(\mathcal M,\overline g)$. For our applications, the manifold $\mathcal M$ will be the product $M\times M$, and $\overline g$ will be the semi--Riemannian metric given by the sum of some semi--Riemannian metric $g$ on $M$ and its opposite $-g$. Consider the inclusion $i:\p\hookrightarrow\mathcal{M}$ of a submanifold $\p\subset\mathcal M$; the restriction $i^*\overline{g}$ may degenerate, in which case the submanifold $\p$ is called {\it degenerate}.

To carry the main tools from Riemannian submanifold theory to the semi--Riemannian context, one is forced to restrict to the nondegenerate case. It is then natural to consider \begin{equation}\label{nondegmet}\met_{\nu}^k(\mathcal{M},\p)=\{\overline{g}\in\met_{\nu}^k(\mathcal{M}): \p \mbox{ is nondegenerate}\}. \end{equation}

\begin{remark}\label{semiriemanniansubmanifold}
If $0<\nu <n$, this subset $\met_{\nu}^k(\mathcal{M},\p)$ might be empty, since there are to\-po\-lo\-gi\-cal obstructions to the existence of semi--Riemannian metrics of fixed index on a compact manifold $\p$. For instance, in the Lorentzian case, if $\p$ is orientable, there exists a Lorentzian metric on $\p$ if and only if $\p$ has Euler cha\-rac\-te\-ris\-tic $0$. In general, $\p$ admits a semi--Riemannian metric of index $\nu$ if and only if it admits a distribution of rank $\nu$.

Characteristic classes, in particular the Euler class, can be used for a more comprehensive study of these obstructions. However, in general this is a fairly difficult problem. For instance, if $\mathcal M$ has dimension $6$ and $\p$ is homeomorphic to a $4$--sphere, then $\met_{3}^k(\mathcal{M},\p)$ is empty. This follows easily from the following facts. On the one hand, the restriction to $\p$ of any metric tensor on $\mathcal M$ having index equal to $3$ cannot be positive or negative definite. On the other hand, $\p$ does not admit any metric tensor of index $1$ or $2$, since $\p$ does not admit smooth distributions of rank $1$ or $2$.\footnote{Recall that a compact manifold
admits a semi-Riemannian metric tensor of index $\nu$ if and only if it admits a smooth distribution of rank $\nu$.}
\end{remark}

If $\overline{g}\in\met_{\nu}^k(\mathcal{M},\p)$, the \emph{second fundamental form} of $\p$ in the normal direction $\eta\in T\p^\perp$ is the symmetric bilinear tensor $\s^\p_\eta\in\sect^k_{\mbox{\tiny sym}}(T\p^*\otimes T\p^*)$, given by \begin{equation}\label{sff} \s^\p_\eta(v,w)=\overline{g}(\nabla^{\overline{g}}_v \overline{w},\eta), \end{equation} where $\nabla^{\overline{g}}$ is the Levi--Civita connection on $(\mathcal{M},\overline{g})$ and $\overline{w}$ is a smooth extension of $w$ tangent to $\p$. Using the fact that $\p$ is nondegenerate, we will also identify $\s^\p_\eta$ at a point $p\in \p$ with the $\overline g$--symmetric linear operator $$\big(\s^\p_\eta\big)_p:T_p \p\longrightarrow T_p\p$$ defined by $\overline g\big(\big(\s^\p_\eta\big)_pv,w\big)=\s^\p_\eta(v,w)$, for all $v,w\in T_p\p$.

\section{An abstract genericity criterion}\label{abstractresult}

In this section we recall a result of Biliotti, Javaloyes and Piccione \cite[Section 3]{biljavapic}, that
 appears also in former paper by Chillingsworth \cite{chillingsworth}, which gives a powerful method to obtain genericity of Morse functionals satisfying appropriate transversality conditions. It follows the lines of a standard transversality argument by White \cite{white}. Recall that a subset of a metric space is said to be \emph{generic} if it contains a dense $G_\delta$, that is, countable intersection of open dense subsets. By the Baire theorem, a generic set is dense.

Assume $Y$ is a Hilbert manifold, and $f_x:Y\rightarrow\R$ is a family of functionals parametrized in an open subset of a Banach manifold $X$. Under suitable regularity hypothesis, the set $\mathfrak{M}=\{(x,y):y \mbox{ is a critical point of } f_x\}$ is an embedded submanifold of $X\times Y$, the projection $\Pi:X\times Y\rightarrow X$ is a nonlinear Fredholm map of index zero and its critical values are precisely the set of parameters $x$ such that $f_x$ has some degenerate critical point\footnote{By \emph{degenerate critical point} of a map we mean a point where the Hessian of this map is not injective. In this case, $y_0\in Y$ such that $d^2f_x(y):T_{y_0}Y\rightarrow T_{y_0}Y^*$ is not injective.} in $Y$. Therefore, the problem of genericity of nondegenerate critical points is reduced to a matter of regular values of a Fredholm map. Applying the Sard--Smale theorem \cite{smale}, one achieves the desired genericity property.

More precisely, the abstract genericity criterion can be stated as follows.

\begin{proposition}\label{abstractgenericity}
Let $f:\mathcal{U}\rightarrow\R$ be a $C^k$ map defined in an open subset $\mathcal{U}\subset X\times Y$, where $X$ is a separable Banach manifold and $Y$ a separable Hilbert ma\-ni\-fold. Assume that for every $(x_0,y_0)\in\mathcal{U}$ such that $\frac{\partial f}{\partial y}(x_0,y_0)=0$, the following conditions hold:

\begin{itemize}
\item[(i)] the Hessian $\frac{\partial^2 f}{\partial y^2}(x_0,y_0):T_{y_0}Y\rightarrow T_{y_0}Y^*\cong T_{y_0}Y$ is a Fredholm operator;
\item[(ii)] for all $w\in\ker\left[\frac{\partial^2 f}{\partial y^2}(x_0,y_0)\right]\setminus\{0\}$, there exists $v\in T_{x_0}X$ such that $$\frac{\partial^2 f}{\partial x \partial y}(x_0,y_0)(v,w)\neq 0.$$
\end{itemize}

Let $\mathcal{U}_x=\{y\in Y:(x,y)\in\mathcal{U}\}$ and denote by $\Pi:X\times Y\rightarrow X$ the projection onto the first factor. Then the set of $x\in X$ such that the functional $$f_x:\mathcal{U}_x\owns y\longmapsto f(x,y)\in\R$$ is a Morse function is generic in the open subset $\Pi(\mathcal{U})\subset X$.
\end{proposition}

\begin{remark}
Given $y_0\in Y$, since $x\mapsto \frac{\partial f}{\partial y}(x,y_0)$ takes values on the fixed Hilbert space $T_{y_0}Y^*$, the mixed derivative in condition (ii) is well defined without the use of a connection on $TY^*$. Also $\frac{\partial^2 f}{\partial y^2}(x_0,y_0)$ is well defined when $\frac{\partial f}{\partial y}(x_0,y_0)=0$, coinciding with the Hessian of $f(x_0,\cdot)$ at the critical point $y_0$.
\end{remark}

\subsection{Brief sketch of the proof}

Recall that a complete proof of such criterion can be found in \cite{biljavapic,chillingsworth}.

Let us briefly give the main lines of how the proof goes. Condition (ii) is a \emph{transversality condition}, more precisely it implies that the map $\tfrac{\partial f}{\partial y}:\mathcal{U}\rightarrow TY^*$ is transversal to the null section of the cotangent bundle $TY^*$. This guarantees that $\mathfrak{M}=\left\{(x,y)\in\mathcal{U}:\tfrac{\partial f}{\partial y}(x,y)=0\right\}$ is an embedded $C^{k-1}$ submanifold of $X\times Y$ and the restriction $\Pi|_\mathfrak{M}$ is a nonlinear $C^{k-1}$ Fredholm map of index zero. Moreover, its critical points are precisely the $(x,y)\in\mathfrak{M}$ such that $y$ is a degenerate critical point of the functional $f_x:\mathcal{U}_x\owns y\mapsto f(x,y)\in\R$. Hence applying the Sard--Smale theorem \cite[Theorem 1.3]{smale} one obtains genericity of parameters $x$ for which $f_x$ has only nondegenerate critical points.

We shall use this abstract criterion in the following set up. The Banach manifold $X$ will be a fixed $C^k$--Whitney type Banach space of tensor fields over $M$, and the Hilbert manifold $Y$ will be the manifold of curves of Sobolev class $H^1$ satisfying a \emph{general boundary condition} on $M$.  Typically, the open subset $\mathcal{U}\subset X\times Y$ will be taken of the form $\mathcal U=\mathcal U_0\times Y$, where $\mathcal U_0$ is an open subset of $X$ consisting of metric tensors. The functional $f$ will be a generalized energy functional. Hence critical points are pairs of metrics and geodesics. The Sard--Smale theorem applied to the projection on the first variable will imply genericity of metrics whose energy functional is Morse. We will give a more precise description of the intended use of this abstract genericity criterion in Section \ref{secgenericity}.

\section{General boundary conditions}\label{sec:gbc}

In this section we study general boundary conditions for the geodesic variational problem on $M$ for curves parametrized in $[0,1]$ with certain regularity. We shall define \emph{general boundary condition} on $M$, analyze the structure of the set of curves satisfying such a boundary condition, and discuss some of its important features, namely the absence of short geodesics under some further \emph{admissibility} assumptions.

A {\it fixed endpoint} boundary condition on $M$ is just a fixed pair of points $(p,q)\in M\times M$, and the correspondent restraint on a curve $\gamma$ is $\gamma(0)=p$ and $\gamma(1)=q$. This is the boundary condition on curves considered in \cite{biljavapic} to prove genericity of metrics without degenerate geodesics, with the assumption that $p\ne q$. Several attempts to generalize this condition are possible, for instance instead of fixing two points $p,q\in M$, fix two submanifolds $P,Q\subset M$, and allow $\gamma(0)\in P$ and $\gamma(1)\in Q$. The most comprehensive generalization is considering a submanifold $\p\subset M\times M$, with the restriction $(\gamma(0),\gamma(1))\in\p$ on the endpoints of curves $\gamma$. This makes arbitrary choices of boundary conditions possible.

\subsection{Nondegeneracy}

Before defining \emph{general boundary condition} and considering the appropriate space of curves satisfying such condition, it is necessary to go into a technical remark on the openness of nondegeneracy property of submanifolds.

Fix an index $\nu\in\{0,\dots,n\}$ and $\mathcal{E}$ a $C^k$--Whitney type Banach space of tensor fields over $M$. Let $\mathcal{A}_\nu\subset\mathcal{E}\cap\met_\nu^k(M)$ be an open subset of the intersection. For any $g\in\mathcal{A}_\nu$, the first step to analyze degeneracy of geodesics with such a boundary condition $\p$ is to induce a metric structure on the submanifold $i:\p\hookrightarrow M\times M$. It will be later clear from Remark \ref{whyg} that the natural choice is to consider the restriction of the ambient space metric $g\oplus (-g)$ to $\p$. Henceforth this \emph{product} metric corresponding to $g\in\mathcal A_\nu$ will be denoted $\overline{g}$.

\begin{proposition}\label{nondegenerateopen} Let $\p\subset M\times M$ be a \emph{compact} submanifold. Then the subset $$\mathcal{A}_{\nu,\p}=\{g\in\mathcal{A}_\nu: \p\subset (M\times M,\overline{g}) \mbox{ is nondegenerate}\}$$ is open in $\mathcal{A}_\nu$.
\end{proposition}

\begin{proof}
Suppose $\mathcal{A}_\nu$ nonempty, otherwise the statement is trivially verified. For each $g\in\mathcal{A}_\nu$, consider the product metric $\overline{g}$. Let $\{g_\alpha\}$ be a convergent sequence in $\mathcal{A}_\nu\setminus\mathcal{A}_{\nu,\p}$ and $\{\overline{g_\alpha}\}$ the correspondent sequence in $\met_n^k(M\times M)\setminus \met_n^k(M\times M,\p)$,\footnote{See \eqref{nondegmet}. Notice that the index of $\overline{g}=g\oplus (-g)$ is always equal to the dimension $n$ of $M$, with no dependence of $\nu$.} with $\lim \overline{g_\alpha}=\overline{g_\infty}$. Identifying at each $p\in \p$ the vector spaces $T_p \p^*\otimes T_p \p^*\cong\lin(T_p \p,T_p \p^*)$, one may consider the symmetric tensor $i^*\overline{g_\alpha}$ at each $p$ as a linear map $$(i^*\overline{g_\alpha})_p:T_p \p\longrightarrow T_p \p^*\cong T_p \p,$$ denoted with the same symbol. Since for all $\alpha$, $i^*\overline{g_\alpha}$ is a degenerate symmetric bilinear tensor on $\p$, there exists $p_\alpha\in\p$ and $V_\alpha\subset T_{p_\alpha} \p$, with $\dim V_\alpha\geq 1$, such that $V_\alpha\subset\ker (i^*\overline{g_\alpha})_{p_\alpha}$. Choosing $r$ to be the minimum of $\dim V_\alpha$, without loss of generality it is possible to assume that for all $\alpha$, $\dim V_\alpha=r\geq 1$.

Thus $\{V_\alpha\}$ is a sequence in the Grassmannian bundle $\gr_r(\p)$, which is compact, since $\p$ is compact. Up to subsequences, there exists $V_\infty\in\gr_r(\p)$ limit of the sequence $\{V_\alpha\}$. By continuity of this convergence, there exists a limit point $p_\infty\in \p$, and $V_\infty\subset\ker (i^*\overline{g_\infty})_{p_\infty}$. Therefore, as $\dim V_\infty=r\geq 1$, the limit metric tensor $\overline{g_\infty}$ is also in $\met_n^k(M\times M)\setminus \met_n^k(M\times M,\p)$, hence $g_\infty\in\mathcal{A}_\nu\setminus\mathcal{A}_{\nu,\p}$.
\end{proof}

\subsection{A few definitions}

Using the same notation from Proposition \ref{nondegenerateopen}, it is now possible to define the following.

\begin{definition}
A \emph{$\nu$--general boundary condition} on $M$ is a \emph{compact} submanifold $\p\subset M\times M$, such that $\mathcal{A}_{\nu,\p}$ is \emph{nonempty} (see Remark \ref{semiriemanniansubmanifold}). When the index $\nu$ is evident from the context, $\p$ will be simply called \emph{general boundary condition}.
\end{definition}

Henceforth, $\p$ will denote a general boundary condition on $M$.

Note that if $\nu=0$, then $\mathcal A_\nu=\mathcal A_{\nu,\p}$ is obviously nonempty for all submanifolds $\p\subset M\times M$. Compactness of $\p$ is a fundamental assumption, not only in order to prove Proposition \ref{nondegenerateopen}, but also because we shall use boundedness of $\p$ to get the desired conditions on limits of curves satisfying such general boundary condition. It is also crucial to consider only nondegenerate metrics, since we shall prove genericity of the set of metrics without degenerate geodesics in $\mathcal{A}_{\nu,\p}$. This is genuinely the natural set of metrics to be considered in this context. Moreover, the submanifold geometry of $\p$ determines the behavior of variational fields correspondent to curves with these conditions, and for instance Lemma \ref{jacobinotparallel} would not hold in case $\p$ was degenerate (see Remark \ref{whyg}).

Let us now investigate the adequate setting for curves on $M$ with endpoints in $\p$. As usual, $H^1([0,1],M)$ denotes the set of all curves of Sobolev class $H^1$ in $M$. It is a well--known fact that $H^1([0,1],M)$ has a canonical Hilbert manifold structure (see Lang \cite{lang} or Palais \cite{palais}) modeled on the separable Hilbert space $H^1([0,1],\R^n)$. In order to verify that the subset \begin{equation}\label{opm}\op(M)=\{\gamma\in H^1([0,1],M):(\gamma(0),\gamma(1))\in\p\},\end{equation} is a Hilbert manifold, consider the {\it double} evaluation map $\ev_{01}:H^1([0,1],M)\rightarrow M\times M$, given by $\ev_{01}(\gamma)=(\gamma(0),\gamma(1))$. It is then easy to see that $\ev_{01}$ is a submersion, hence $\op(M)=\ev_{01}^{-1}(\p)$ is a submanifold of $H^1([0,1],M)$. Furthermore, the tangent space $T_\gamma\op(M)$ can be identified with the Hilbertable space of all sections $v$ of Sobolev class $H^1$ of the pull--back bundle $\gamma^*TM$ such that $(v(0),v(1))\in T_{(\gamma(0),\gamma(1))}\p$.

Moreover, for each $\gamma\in\op(M)$, the fixed complete Riemannian metric $g_\mathrm{R}$ on $M$ induces a Riemannian structure on the fibers of the pull--back bundle $\gamma^*TM$. Hence $\op(M)$ can be endowed with a Riemann--Hilbert structure using the inner product in $T_\gamma\op(M)$ given by \begin{equation}\label{riemhilbop}\langle v,w\rangle =\int_0^1 g_\mathrm{R}(\D^{g_\mathrm{R}}v,\D^{g_\mathrm{R}}w)\;\mathrm{d}t.\end{equation}

\begin{example}\label{exgbc}
The fixed endpoints condition $\p=\{p\}\times\{q\}$ is \emph{trivially}\footnote{Note that $\p$ is automatically nondegenerate, since the tangent space to $\p$ is trivial and every possible ambient metric induces the identically null metric on $\p$. This also holds if $p=q$.} a general boundary condition. As expected, the tangent space $T_\gamma\op(M)$ is formed by Sobolev class $H^1$ sections $v$ of $\gamma^*TM$ such that $v(0)=0$ and $v(1)=0$. Similarly, if $P$ and $Q$ are compact submanifolds, then $\p=P\times Q$ is a general boundary condition, unless it fails to admit semi-Riemannian metrics of the appropriate index (see Remark \ref{semiriemanniansubmanifold}). The curves $\gamma\in\op(M)$ satisfy $\gamma(0)\in P$ and $\gamma(1)\in Q$, and the condition on the sections $v$ of $\gamma^*TM$ that form the tangent space is $v(0)\in T_{\gamma(0)}P$ and $v(1)\in T_{\gamma(1)}Q$. Note that $\p$ would still be a general boundary condition if one submanifold was taken as a point, i.e., $Q=\{q\}$.
\end{example}

Note that the \emph{transpose} of a general boundary condition $\p$, defined by \begin{equation}\label{transpose}\p^t=\{(p,q)\in M\times M:(q,p)\in\p\},\end{equation} is also a general boundary condition, and the spaces $\op(M)$ and $\Omega_{\p^t}(M)$ can be canonically identified by reparametrizing curves using the backwards parameterization. Hence solutions of the geodesic variational problems with boundary conditions $\p$ and $\p^t$ are also obviously identified. Due to such symmetry, every result stated for some general boundary condition $\p$ is also automatically valid for its transpose $\p^t$.

\begin{definition}
Fix $g\in\mathcal A_{\nu,\p}$. A $g$--geodesic $\gamma\in\op(M)$ will be called a \emph{$(g,\p)$--geodesic} if it satisfies $$(\dot\gamma(0),\dot\gamma(1))\in T_{(\gamma(0),\gamma(1))}\p^\perp,$$ where $^\perp$ denotes orthogonality relatively to $\overline g$. It will be seen in Section \ref{secgenericity} that this is equivalent to $(g,\gamma)$ being a critical point of a generalized energy functional.
\end{definition}

\subsection{Periodic geodesics}\label{sec:pg}

An interesting example of boundary condition for geo\-de\-sics is given by the dia\-go\-nal\footnote{Here $\Delta\subset M\times M$ is the diagonal of the product manifold $M\times M$, however in the sequel we will be somewhat sloppy about the use of the symbol $\Delta$. It will denote the diagonal not only of $M\times M$, but also of any product space, for instance $\Delta$'s own tangent space, which is the diagonal $\Delta\subset T_xM\oplus T_xM$. There is no ambiguity, since it will always be clear from the context which diagonal is being considered.} $$\Delta=\{(p,p):p\in M\};$$ critical points of the $g$--energy functional in the space of curves with endpoints on $\Delta$ are \emph{periodic} $g$--geodesics. Note however that $\Delta$, more generally any submanifold $\mathcal P$ somewhere tangent to $\Delta$, is always degenerate for a metric of the form $\overline g=g\oplus(-g)$. Thus, these are not general boundary conditions. Indeed, if $\p$ is tangent to $\Delta$ at $(p,p)$, the tangent space $T_{(p,p)}\p$ is a subspace of the diagonal $\Delta\subset T_pM\oplus T_pM$, hence $\overline{g}$ is identically null in this space. In particular, $\Delta\subset M\times M$ itself is \emph{not} a general boundary condition.

Nevertheless, Biliotti, Javaloyes and Piccione \cite{biljavapic2} recently managed to use equivariant variational
genericity to prove the \emph{semi--Riemannian bumpy metric theorem}\footnote{This theorem asserts that the set of \emph{bumpy} metrics, that is, metrics without degenerate periodic geodesics, is generic. The Riemannian version of this result was formulated by Abraham \cite{abraham} and proved by Anosov \cite{anosov}.}, which corresponds to our main result, Theorem \ref{bigone}, in case $\p=\Delta$. Our technique though does not apply to this case, and we \emph{use} the semi--Riemannian bumpy metric theorem to prove our genericity statement if $\p\cap\Delta\neq\varnothing$.

\subsection{Admissibility}

For the main result, it is necessary to have a lower bound on the Riemannian length of nonconstant $(g,\p)$--geodesics. To this aim we introduce the following.

\begin{definition}\label{def:admgbc}
A $\nu$--general boundary condition $\p$ will be said to be \emph{admissible} if for every $g_0\in\mathcal A_{\nu,\p}$, there exists an open neighborhood $\mathcal V$ of $g_0$ in $\mathcal A_{\nu,\p}$ and $a>0$, such that for all $g\in\mathcal V$, and all $(g,\p)$--geodesics $\gamma$, $L_\mathrm R(\gamma)\geq a$.
\end{definition}

It is easy to see that this definition does not depend on the choice of the Rie\-man\-nian metric $g_\mathrm R$. Some elementary classes of admissible general boundary conditions are worth mentioning. Firstly, if the general boundary condition $\p$ satisfies $\p\cap\Delta=\varnothing$, then it is admissible. In this case, it is enough to set $$a=\min_{(p,q)\in\p} d_\mathrm R(p,q),$$ where $d_\mathrm R$ denotes the $g_\mathrm R$--distance in $M$.

Another class of admissible general boundary conditions is given by $\p=P\times \{q\}$, where $P\subset M$ is a compact submanifold, and $q\in M$, as in Example \ref{exgbc}. There are two possible situations; namely if $q\notin P$, then $\p\cap\Delta =\varnothing$, hence it is also in the previous class. However, if $q\in P$, the proof of \cite[Lemma 3.6]{biljavapic2} can be used to verify that $\p$ is admissible. In fact, although stated only for periodic geodesics, its proof is automatically valid considering nonconstant geodesic loops instead of periodic geodesics, hence gives the required condition on $\p$. Note that the same is true for the transpose $\p^t=\{q\}\times P$.

We shall now establish the admissibility of a larger class of general boundary conditions that intersect $\Delta$, using again a transversality approach. To this aim, we give an estimation of the decrease of the difference between the normalized tangent field to a geodesic at its endpoints, in terms of its length.

\begin{lemma}\label{short}
Let $U\subset\R^n$ be an open subset and $g_0\in\met_\nu^k(U)$. Then for all compact subsets $K\subset U$ there exists a positive number $c>0$ and an open neighborhood $\mathcal O$ of $g_0$ in the weak Whitney $C^1$--topology, such that for all $g\in\mathcal O$ and all nonconstant $g$--geodesic $\gamma:[a,b]\rightarrow U$ with $\gamma([a,b])\subset K$, the following inequality holds \begin{equation} \left\| \frac{\dot\gamma(b)}{\|\dot\gamma(b)\|}-\frac{\dot\gamma(a)}{\|\dot\gamma(a)\|}\right\|\leq c\int_a^b \|\dot\gamma(t)\|\;\mathrm{d}t,\end{equation} where $\|\cdot\|$ is the Euclidean norm.
\end{lemma}

\begin{proof}
Given $g\in\met_\nu^k(U)$, denote by $\Gamma^g$ the Christoffel tensor of $g$ relatively to the Euclidean metric on $U$. Thus, for all $x\in U$, $\Gamma^g(x):\R^n\times\R^n\rightarrow\R^n$ is a symmetric bilinear map depending continuously on $x$, and if $\gamma$ is a $g$--geodesic, $\ddot\gamma=\Gamma^g(\gamma)(\dot\gamma,\dot\gamma)$, where $\ddot\gamma$ denotes the ordinary second derivative of $\gamma$ in $\R^n$. This association $g\mapsto\Gamma^g$ is clearly continuous when $\met_\nu^k(U)$ is endowed with the weak Whitney $C^1$--topology and the space of $\Gamma^g$'s is endowed with the weak Whitney $C^0$--topology. If $K\subset U$ is a given compact subset, set $\kappa=\max_{x\in K} \|\Gamma^{g_0}(x)\|+1$ and define $$\mathcal O=\{g\in\met_\nu^k(U):\|\Gamma^g(x)\| <\kappa, \; \forall  x\in K\},$$ which is obviously an open neighborhood of $g_0$ in the weak Whitney $C^1$--topology.

Let us show that such $\mathcal O$ satisfies the thesis, with $c=2\kappa$. Indeed, if $g\in\mathcal O$ and $\gamma$ is a nonconstant $g$--geodesic with image lying in $K$, then at each time $t\in [a,b]$, \begin{multline*}
\left\|\frac{d}{dt}\frac{\dot\gamma}{\|\dot\gamma\|}\right\| =\left\|\frac{\ddot\gamma}{\|\dot\gamma\|}-\frac{\dot\gamma\langle\dot\gamma,\ddot\gamma\rangle}{\|\dot\gamma\|^3}\right\| =\left\|-\frac{\Gamma^g(\gamma)(\dot\gamma,\dot\gamma)}{\|\dot\gamma\|}+\frac{\langle\dot\gamma,\Gamma^g(\gamma)(\dot\gamma,\dot\gamma)\rangle}{\|\dot\gamma\|^3}\dot\gamma\right\| \leq \\ \leq \frac{\|\Gamma^g(\gamma)\|\|\dot\gamma\|^2}{\|\dot\gamma\|} + \frac{\|\Gamma^g(\gamma)\|\|\dot\gamma\|^4}{\|\dot\gamma\|^3}\leq 2\kappa\|\dot\gamma\|. \end{multline*} Integrating the above inequality in $[a,b]$, it follows that $$\left\| \frac{\dot\gamma(b)}{\|\dot\gamma(b)\|}-\frac{\dot\gamma(a)}{\|\dot\gamma(a)\|}\right\|\leq\left\|\int_a^b \frac{d}{dt}\frac{\dot\gamma}{\|\dot\gamma\|}\;\mathrm{d}t\right\|\leq \int_a^b \left\|\frac{d}{dt}\frac{\dot\gamma}{\|\dot\gamma\|}\right\|\;\mathrm{d}t\leq 2\kappa\int_a^b\|\dot\gamma(t)\|\;\mathrm{d}t.\qedhere$$
\end{proof}

\begin{proposition}\label{admissibility}
If a $\nu$--general boundary condition $\p$ intersects $\Delta$ transversally\footnote{That is, $T_{(x,x)}\p +\Delta=T_xM\oplus T_xM$, for all $x\in\p\cap\Delta$.}, then $\p$ is admissible.
\end{proposition}

\begin{proof}
We proceed by contradiction. Since the weak Whitney $C^1$--topology is first coun\-ta\-ble, assuming $\p$ is not admissible implies that there exists a sequence $\{g_\alpha\}$ in $\mathcal A_{\nu,\p}$ converging to some $g_0\in\mathcal A_{\nu,\p}$ in the weak Whitney $C^1$--topology and a sequence $\gamma_\alpha\in\op(M)$ of nonconstant $(g_\alpha,\p)$--geodesics such that $\lim L_\mathrm R(\gamma_\alpha)=0$. Since $\p$ is compact, up to taking subsequences, one can assume that there exists $x\in M$ such that $(x,x)\in\p$ and both $\lim \gamma_\alpha(0)=x$, $\lim \gamma_\alpha(1)=x$.

By taking a local chart of $M$ around $x$, we can assume that we are in open subset $U\subset \R^n$. Let $K\subset U$ be any compact neighborhood of $x$, so that there exists $\alpha_0$ such that for $\alpha\geq\alpha_0$, $\gamma_\alpha([0,1])\subset K$. Since $L_\mathrm R(\gamma_\alpha)$ tends to zero, then also the Euclidean length of $\gamma_\alpha$ tends to zero. From Lemma \ref{short}, it follows that, $$\lim \left(\frac{\dot\gamma_\alpha(0)}{\|\dot\gamma_\alpha(0)\|} -\frac{\dot\gamma_\alpha(1)}{\|\dot\gamma_\alpha(1)\|}\right)=0,$$ and up to taking subsequences, we can assume that both $\frac{\dot\gamma_\alpha(0)}{\|\dot\gamma_\alpha(0)\|}$ and $\frac{\dot\gamma_\alpha(1)}{\|\dot\gamma_\alpha(1)\|}$ converge to unitary vectors. However, from the above limit, both tend to the \emph{same unitary vector $v\in\R^n$}.

We claim that $(v,v)\in T_{(x,x)}\p^\perp$, where $^\perp$ denotes orthogonality with respect to $\overline{g_0}$, and that this concludes the proof. Indeed, suppose the claim to be true. Then $$(v,v)\in T_{(x,x)}\p^\perp \cap\Delta=(T_{(x,x)}\p + \Delta^\perp)^\perp.$$ It is easy to see that $\Delta^\perp=\Delta$; and since we assumed $T_{(x,x)}\p +\Delta=T_xM\oplus T_xM$, its orthogonal complement with respect to $\overline{g_0}$ is trivial. Hence $v=0$, which gives the desired contradiction.

It remains to prove the above claim that $(v,v)\in T_{(x,x)}\p^\perp$. Consider $\mathcal O$ the open neighborhood\footnote{Using the identification above given by a local chart $U$ of $M$ around $x$, since the restriction map $\mathcal A_{\nu,\p}\owns g\mapsto g|_U\in\met_\nu^k(U)$ is continuous in the considered topologies, the open neighborhood of $g_0$ in $\mathcal A_{\nu,\p}$ can be taken as the preimage of $\mathcal O$ by this restriction map.} of $g_0\in\met_\nu^k(U)$ in the weak Whitney $C^1$--topology given by Lemma \ref{short} with the choices above. Then, for all $g\in\mathcal O$ it is possible to give the following estimation for any $g$--geodesic $\gamma$ with image lying in $K$, $$\left|\frac{d}{dt} \log \|\dot\gamma(t)\|\right| = \frac{\left|\langle\dot\gamma,\ddot\gamma\rangle\right|}{\|\dot\gamma\|^2} = \frac{\left|\langle\dot\gamma,\Gamma^{g}(\gamma)(\dot\gamma,\dot\gamma)\rangle\right|}{\|\dot\gamma\|^2}\leq \frac{\|\Gamma^g(\gamma)\|\|\dot\gamma\|^3}{\|\dot\gamma\|^2}\leq \kappa\|\dot\gamma\|,$$ where $\kappa=\max_{x\in K} \|\Gamma^{g_0}(x)\| +1$ is again the same as in Lemma \ref{short}. Hence, integrating the above inequality in $[0,1]$, it follows that $$\left|\log\frac{\|\dot\gamma(1)\|}{\|\dot\gamma(0)\|}\right|=\left|\int_0^1 \frac{d}{dt}\log\|\dot\gamma\|\;\mathrm{d}t\right|\leq\int_0^1 \left|\frac{d}{dt}\log\|\dot\gamma\|\right|\;\mathrm{d}t\leq\kappa\int_0^1\|\dot\gamma\|\;\mathrm{d}t.$$

Applying this estimation to the $(g_\alpha,\p)$--geodesics $\{\gamma_\alpha\}$, since its Euclidean length tend to zero, one concludes that $$\lim \frac{\|\dot{\gamma_\alpha}(1)\|}{\|\dot{\gamma_\alpha}(0)\|}=1.$$ Moreover, for each $\alpha$, $$\left(\frac{\dot{\gamma_\alpha}(0)}{\|\dot{\gamma_\alpha}(1)\|},\frac{\dot{\gamma_\alpha}(1)}{\|\dot{\gamma_\alpha}(1)\|}\right)\in T_{(\gamma_\alpha(0),\gamma_\alpha(1))}\p^{\perp_\alpha},$$ where $^{\perp_\alpha}$ denotes orthogonality with respect to $\overline{g_\alpha}$, and $\lim \frac{\dot{\gamma_\alpha}(1)}{\|\dot{\gamma_\alpha}(1)\|}=v$.

From $\lim \frac{\dot{\gamma_\alpha}(0)}{\|\dot{\gamma_\alpha}(0)\|}=v$ and $\lim \frac{\|\dot{\gamma_\alpha}(1)\|}{\|\dot{\gamma_\alpha}(0)\|}=1$, it follows that also $\lim \frac{\dot{\gamma_\alpha}(0)}{\|\dot{\gamma_\alpha}(1)\|}=v$. Since $\p$ is compact, this proves the claim that $(v,v)\in T_{(x,x)}\p^\perp$.
\end{proof}

\begin{remark}
Since admissibility of $\p$ can be characterized by its transversality to $\Delta$, it follows that admissibility is a generic property of general boundary conditions.
\end{remark}

To end this section, we analyze admissibility of some general boundary conditions given in Example \ref{exgbc}.

\begin{example}\label{adexgbc}
The fixed endpoints boundary condition $\p=\{p\}\times\{q\}$ is clearly admissible, even if $p=q$. Indeed, it falls in the class of boundary conditions of the form $\p=P\times\{q\}$, where $P\subset M$ is a compact submanifold, explored in the beginning of this subsection. Substituting $q$ for a compact submanifold $Q\subset M$ gives $\p=P\times Q$, as in Example \ref{exgbc}. This is also clearly an admissible general boundary condition if $P\cap Q=\varnothing$. If $P\cap Q\neq\varnothing$, it is easy to see that $\p$ is transversal to $\Delta$ if and only if $P$ and $Q$ are transversal submanifolds of $M$.
\end{example}

\section{Genericity of metrics without degenerate geodesics}\label{secgenericity}

In this section we prove our main result, the genericity of semi--Riemannian metrics without degenerate geodesics satisfying an admissible general boundary condition. It is an immediate generalization of the genericity result in \cite{biljavapic} that corresponds to $\{p\}\times\{q\}$ to any admissible general boundary condition $\p$, even if $\p\cap\Delta\neq\varnothing$.

More precisely, consider again $M$ a $n$--dimensional smooth manifold, an index $\nu\in\{0,\dots,n\}$ and  $\mathcal{A}_\nu\subset\mathcal{E}\cap\met_\nu^k(M)$ a nonempty open subset of $\mathcal{E}$, which is a fixed $C^k$--Whitney type Banach space of tensor fields over $M$. Consider also $\p$ an admissible $\nu$--general boundary condition, $g_\mathrm{R}$ the fixed complete Riemannian metric on $M$ and the Hilbert--Riemann structure it induces on $\op(M)$ (see \eqref{opm}), given by (\ref{riemhilbop}). We shall prove that the set of semi--Riemannian metrics on $M$ of fixed index $\nu$ such that all $(g,\p)$--geodesics\footnote{Recall that by \emph{geodesic} we mean \emph{affinely parametrized} geodesic.} are nondegenerate is generic in $\mathcal{A}_{\nu,\p}$ (see Proposition \ref{nondegenerateopen}).

For this we shall use the abstract genericity criterion given in Proposition \ref{abstractgenericity} with the following geodesic setup. The Banach manifold $X$ will be taken as the Banach space $\mathcal{E}$ and the Hilbert manifold $Y$ as $\op(M)$. The open subset of $X\times Y$ is $\mathcal{U}=\mathcal{A}_{\nu,\p}\times\op(M)$, domain of the generalized energy functional \begin{equation}\label{efunct} f:\mathcal{U}\owns (g,\gamma)\longmapsto\tfrac{1}{2}\int_0^1 g(\dot{\gamma},\dot{\gamma}) \;\mathrm{d}t\in\R, \end{equation} which is a $C^k$ functional. More precisely, it is smooth with respect to the first variable $g\in\mathcal{A}_{\nu,\p}$ and $C^k$ with respect to the second variable $\gamma$. Furthermore, with this formulation, $(g_0,\gamma_0)$ is a critical point of $f$ if and only if $\gamma_0$ is a $(g_0,\p)$--geodesic, i.e., $\gamma_0\in\op(M)$ is a $g_0$--geodesic and $$(\dot{\gamma_0}(0),\dot{\gamma_0}(1))\in T_{(\gamma_0(0),\gamma_0(1))}\p^\perp,$$ where $^\perp$ denotes orthogonality with respect to $\overline{g_0}$.

Hence, critical points of the projection $\Pi:\mathcal U\rightarrow \mathcal{A}_{\nu,\p}$ correspond to the set of $(g_0,\gamma_0)$ such that $\gamma_0$ is a degenerate\footnote{That is, $\frac{\partial^2 f}{\partial \gamma^2}(g_0,\gamma_0)$ is not injective. If $\p=\{p\}\times\{q\}$ is a fixed endpoints condition, this means that $p$ and $q$ are conjugate along $\gamma_0$. For a general $\p$, we shall characterize the elements of the index form kernel as \emph{$\p$--Jacobi fields} and give further geometric interpretation in Subsection \ref{bigsec}.} $(g_0,\p)$--geodesic. Thus applying Proposition \ref{abstractgenericity} we shall conclude that the set of $g\in\mathcal{A}_{\nu,\p}$ such that the $g$--energy functional is Morse is generic in $\Pi(\mathcal U)=\mathcal{A}_{\nu,\p}$.

\subsection{Derivatives of the energy functional}\label{derivativesoff}

In order to verify the hypothesis of the abstract genericity criterion, we first compute the index form of the generalized energy functional and its kernel, verifying condition (i); and we also calculate the mixed derivative correspondent to condition (ii).

An easy computation gives the following formula for the index form \begin{multline}\label{indexform}
\frac{\partial^2 f}{\partial \gamma^2}(g_0,\gamma_0)(v,w) = \int_0^1 g_0(\D^{g_0} v,\D^{g_0} w)-g_0(R^{g_0}(\dot{\gamma_0},v)w,\dot{\gamma_0})\;\mathrm{d}t \\ -\s^\p_{(\dot{\gamma_0}(0),\dot{\gamma_0}(1))}\Big((v(0),v(1)),(w(0),w(1))\Big), \end{multline} where $\s^\p_\eta$ is the second fundamental form of $\p$ with normal $\eta\in T\p^\perp$, with respect to the ambient metric $\overline{g_0}$.

\begin{lemma}\label{fredholmness}
The index form $\frac{\partial^2 f}{\partial \gamma^2}(g_0,\gamma_0)$ is a {\rm Fredholm} symmetric bilinear form on $T_{\gamma_0}\op(M)$, i.e., it is represented by a self--adjoint Fredholm operator of this Hilbert space.
\end{lemma}

\begin{proof}
For each $t\in [0,1]$, denote by $A_t$ the automorphism of $T_{\gamma(t)}M$ that represents $g_0$ in terms of the fixed Riemannian metric $g_\mathrm{R}$, that is, such that $g_0=g_\mathrm{R}(A_t\cdot,\cdot)$. Then the map $\Phi:T_{\gamma_0}\op(M)\rightarrow T_{\gamma_0}\op(M)$ that
carries $v$ to $\tilde v$, where  $\tilde{v}(t)=A_tv(t)$, is an isomorphism.

We shall prove that the index form \eqref{indexform}, is a compact perturbation of $\Phi$, hence Fredholm. Indeed, the difference $D(v,w)=\frac{\partial^2 f}{\partial \gamma^2}(g_0,\gamma_0)(v,w)-\langle\Phi v,w\rangle$ is given by
\begin{multline*}D(v,w) = \int_0^1 \Big[-g_\mathrm R(A'v,\D^\mathrm{R}w)+g_\mathrm{R}(A\D^\mathrm{R}v,\chr^\mathrm{R}w)+g_\mathrm{R}(A\chr^\mathrm{R}v,\D^\mathrm{R}w) \\  +g_\mathrm{R}(A\chr^\mathrm{R}v,\chr^\mathrm{R}w)+g_\mathrm{R}(AR(v),w)\Big]\;\mathrm{d}t \\  -\s^\p_{(\dot{\gamma_0}(0),\dot{\gamma_0}(1))}\Big((v(0),v(1)),(w(0),w(1))\Big), \end{multline*} where $\chr^\mathrm{R}=\D^{g_0}-\D^\mathrm{R}$ is the Christoffel tensor of $\nabla^{g_0}$ relatively to $\nabla^\mathrm{R}$, $R(v)=R^{g_0}(\dot{\gamma_0},v)\dot{\gamma_0}$ and $A'$ is the covariant derivative\footnote{$A$ can be thought as a $C^k$ section of $\gamma_0^*(TM^*\otimes TM)$, and the connection $\nabla^\mathrm{R}$ induces a canonical connection on this bundle.} of $A$.

Note that each term of the integral above is a bilinear form in $T_{\gamma_0}\op(M)\times T_{\gamma_0}\op(M)$ that does not contain more than one derivative of its arguments. Hence each term is continuous in one of its arguments, and continuous in the $C^0$--topology in the other\footnote{Using the inclusion $H^1\hookrightarrow C^0$ it is possible to induce a \emph{$C^0$--topology} in $T_{\gamma_0}\op(M)$.}; and since the inclusion $H^1\hookrightarrow C^0$ is compact, each of these bilinear forms is represented by a compact operator of $T_{\gamma_0}\op(M)$. Furthermore, the last term of the expression above for $D$ is also represented by a compact operator of $T_{\gamma_0}\op(M)$, since it is the image by an evaluation map with values on a finite dimensional vector space of a $C^0$--continuous bilinear form in $T_{\gamma_0}\op(M)$. Therefore $D$ is represented by a compact operator of $T_{\gamma_0}\op(M)$.
\end{proof}

Moreover, the kernel of the index form $\frac{\partial^2 f}{\partial \gamma^2}(g_0,\gamma_0)$ is the space of all Jacobi fields $J\in T_{\gamma_0}\op(M)$, such that \begin{equation}\label{pcampodijacobi}(\D^{g_0} J(0),\D^{g_0} J(1))+\s^\p_{(\dot{\gamma_0}(0),\dot{\gamma_0}(1))}\Big(J(0),J(1)\Big)\in T_{(\gamma_0(0),\gamma_0(1))}\p^\perp. \end{equation} The elements of this space will be called {\it $\p$--Jacobi fields}.

\begin{example}\label{exjac}
According to expected, in the cases of admissible general boundary conditions given in Example \ref{adexgbc}, $\p=\{p\}\times\{q\}$ and $\p=P\times Q$, the $\p$--Jacobi fields are Jacobi fields along $\gamma_0$ that, respectively, vanish at the endpoints, and $J(0)\in T_{\gamma_0(0)}P$ and $J(1)\in T_{\gamma_0(1)}Q$. Geometrically, existence of a nontrivial $\p$--Jacobi field in the previous cases can be interpreted as follows. In the first case, it simply means that $p$ and $q$ are conjugate along $\gamma_0$; and in the second, if $Q=\{q\}$ is a point, it means that $q$ is \emph{focal} to $P$. For the geometrical interpretation of conjugacy for general products
$\p=P\times Q$ see for instance \cite{PicTauJMP}.
\end{example}

\begin{remark}\label{tangentaintpjacobi}
Consider $\gamma\in\op(M)$ a $(g,\p)$--geodesic. Although the tangent field $\dot\gamma$ is a Jacobi field, it is \emph{never} a \emph{$\p$--Jacobi field}. This is immediate from the fact that $\dot\gamma$ is $\overline{g}$--orthogonal to $\p$ at $(\gamma(0),\gamma(1))$, and that all $\p$--Jacobi fields along $\gamma$ must be tangent to $\p$ at this point. Since $\overline{g}$ does not degenerate on $\p$, it follows $\dot\gamma\notin T_\gamma\op(M)$ and hence $\dot\gamma$ cannot be a $\p$--Jacobi field.

Note that this observation includes the case of geodesics loops, which may be $(g,\p)$--geodesics if $\p\cap\Delta\neq\varnothing$. In addition, it also covers the possibility $\p=\{p\}\times\{q\}$, even if $p=q$. In such case, the tangent space $T_{(p,q)}\p$ is trivial, hence all $\p$--Jacobi fields $J$ along $\gamma$ have to satisfy $J(0)=0$ and $J(1)=0$. Therefore, $\dot\gamma$ is not a $\p$--Jacobi field once more.
\end{remark}

\begin{remark}\label{degeneracynotions}
Suppose $\p\cap\Delta\neq\varnothing$ and let $\gamma$ be a periodic $g$--geodesic that is also a $(g,\p)$--geodesic. As a consequence of Remark \ref{tangentaintpjacobi}, the notions of degeneracy of $\gamma$ differ when it is considered as a \emph{periodic geodesic} and as a \emph{$(g,\p)$--geodesic}. More precisely, the tangent field $\dot\gamma$ is always a Jacobi field along $\gamma$, therefore $\gamma$ would always be a degenerate critical point of the $g$--energy functional. Such degeneracy is caused by the obvious action of the circle $S^1$ on periodic curves, by right composition. To treat this special case, one is forced to use an equivariant definition of degeneracy. Namely, a periodic geodesic $\gamma$ is said to be \emph{degenerate as a periodic geodesic} if it admits a periodic Jacobi field that is not a constant multiple of $\dot\gamma$.

Since the tangent field $\dot\gamma$ is not a $\p$--Jacobi field along $\gamma$, it follows that if $\gamma$ is nondegenerate as a periodic geodesic, then it is also nondegenerate as a $(g,\p)$--geodesic. However, the converse is not true, since $\gamma$ may admit a Jacobi field which is not a constant multiple of $\dot\gamma$, neither a $\p$--Jacobi field.
\end{remark}

Let $\gamma$ be a $(g,\p)$--geodesic. Not only the tangent field $\dot\gamma$ is not a $\p$--Jacobi field (Remark \ref{tangentaintpjacobi}), but also $\p$--Jacobi fields along $\gamma$ are only parallel to $\dot\gamma$ at a finite number of points. Such claim is a consequence of Lemma \ref{parallelfinite} combined with the following result.

\begin{lemma}\label{jacobinotparallel}
Let $\gamma:[a,b]\rightarrow M$ be a $g$--geodesic. If $J$ is a nontrivial $\p$--Jacobi field along $\gamma$, then it is not everywhere parallel to $\dot{\gamma}$.
\end{lemma}

\begin{proof}
Firstly, let us consider the trivial case when $\p$ is not a point. Since $J$ is a $\p$--Jacobi field, $(J(a),J(b))\in T_{(\gamma(a),\gamma(b))}\p$. Hence $J(a)$ and $J(b)$ are not respectively parallel to $\dot{\gamma}(a)$ and $\dot{\gamma}(b)$, because $(\dot{\gamma}(a),\dot{\gamma}(b))\in T_{(\gamma(a),\gamma(b))}\p^\perp$.

If $\p=\{p\}\times\{q\}$, the argument is modified as follows. In this case, suppose that there exists $\lambda:[a,b]\rightarrow M$ such that $J(t)=\lambda(t)\dot{\gamma}(t)$. Since $J$ is the solution of the Jacobi equation, $\lambda$ must be an affine function, that is, $\lambda(t)=c_1+c_2t$ for some $c_1,c_2\in\R$. Moreover, $(J(a),J(b))$ is tangent and orthogonal to $\p$ at $(\gamma(a),\gamma(b))$, thus $\lambda(a)=\lambda(b)=0$. This implies that $J$ is the trivial solution.
\end{proof}

\begin{remark}\label{whyg}
This is the reason to choose the product metric $\overline{g}=g\oplus (-g)$ instead of any other. Note that if the metric in $\p$ was different, it would be possible that the tangent field $\dot\gamma$ was a $\p$--Jacobi field. Furthermore, notice that the nondegeneracy of $\p$ with respect to this $\overline{g}$ is essential in the proof.
\end{remark}

To compute the second mixed derivative $\frac{\partial^2 f}{\partial g\partial \gamma}$ of the energy functional \eqref{efunct}, it is convenient to use Schwartz lemma. Since the domain $\mathcal{A}_{\nu,\p}\times\op(M)$ is the product of an open subset $\mathcal{A}_{\nu,\p}$ of a Banach space $\mathcal{E}$ and a Hilbert manifold $\op(M)$, the first partial derivative can be thought as $\frac{\partial f}{\partial g}:\mathcal{A}_{\nu,\p}\times\op(M)\rightarrow\mathcal{E}^*$.

Fix $g_0\in\mathcal{A}_{\nu,\p}$. Deriving $\frac{\partial f}{\partial g}(g_0,\cdot)$, one obtains \begin{equation}\label{yx}\frac{\partial}{\partial \gamma}\frac{\partial f}{\partial g}(g_0,\gamma_0):T_{\gamma_0}\op(M)\longrightarrow\mathcal{E}^*,\end{equation} which may also be seen as a bilinear form on $T_{\gamma_0}\op(M)\times\mathcal{E}$. If instead of deriving $f$ first in $g$, one derives first in $\gamma$ and then in $g$, the result is \begin{equation}\label{xy}\frac{\partial}{\partial g}\frac{\partial f}{\partial \gamma}(g_0,\gamma_0):\mathcal{E}\rightarrow T_{\gamma_0}\op(M)^*, \end{equation} which is a bilinear form on $\mathcal{E}\times T_{\gamma_0}\op(M)$. Using local charts and Schwartz lemma, it follows that these maps are transpose to each other, that is, for all $(v,w)\in\mathcal{E}\times T_{\gamma_0}\op(M)$, $\frac{\partial^2 f}{\partial g\partial\gamma}(g_0,\gamma_0)(v,w)=\frac{\partial^2 f}{\partial\gamma\partial g}(g_0,\gamma_0)(w,v)$.

We are interested in computing \eqref{xy}, however it turns out to be easier to compute \eqref{yx}, so we shall use the observation above. Since $f$ is linear in the first variable, for all $h\in\mathcal{E}$, $$\label{firstderivative}\frac{\partial f}{\partial g}(g_0,\gamma)h=\tfrac{1}{2}\int_0^1 h(\dot{\gamma},\dot{\gamma})\;\mathrm{d}t.$$ Fix any\footnote{It is easy to see that the following construction does not depend on the choice of $\nabla$.} symmetric connection $\nabla$ on $M$. Deriving $\frac{\partial f}{\partial g}(g_0,\cdot)$, one obtains for each $(h,v)\in\mathcal{E}\times T_{\gamma_0}\op(M)$, \begin{equation}\label{mixedderivative} \frac{\partial^2 f}{\partial\gamma\partial g}(g_0,\gamma_0)(v,h)=\int_0^1 h(\dot{\gamma}_0,\D v)+\tfrac{1}{2}\nabla h(v,\dot{\gamma_0},\dot{\gamma_0})\;\mathrm{d}t, \end{equation} where $\D$ is the covariant derivative of vector fields along $\gamma_0$ induced by $\nabla$. This gives a final formula for the mixed derivative which is crucial to verify condition (ii) of Proposition \ref{abstractgenericity}.

\subsection{\emph{Strongly degenerate} geodesics}\label{sec:stdeg}

Before proving our main theorem, let us introduce a class of geodesics that will play a special role in the final arguments for periodic geodesics.

\begin{definition}\label{defstdeg}
Let $\gamma:[0,1]\rightarrow M$ be a $g$--geodesic. Then $\gamma$ is said to be \emph{strongly degenerate} if
there exists an integer $k\ge2$ such that:
\begin{itemize}
\item[(a)] $\gamma\left(t+\tfrac{i}{k}\right)=\gamma(t)$, for all $i\in\{0,\dots,k-1\}$ and $t\in \left[0,\tfrac{1}{k}\right[$;
\item[(b)] $\gamma$ admits a Jacobi field $J\ne0$, such that $\sum\limits_{i=0}^{k-1} J\left(t+\tfrac{i}{k}\right)=0$, for all $t\in \left[0,\tfrac{1}{k}\right[$.
\end{itemize}
\end{definition}

Observe that if $\gamma$ is strongly degenerate, then it is automatically a periodic geodesic with period $\tfrac{1}{k}$. Moreover, our next result asserts it is also degenerate as such.

\begin{proposition}\label{stronglydegenerate}
If $\gamma\in\op(M)$ is a strongly degenerate $(g,\p)$--geodesic, then it is also degenerate as a periodic geodesic (see Remark \ref{degeneracynotions}). That is, $\gamma$ admits a nontrivial periodic Jacobi field $J$ that is not a constant multiple of $\dot\gamma$.
\end{proposition}

\begin{proof}
Take $J$ a Jacobi field as in (b). Then $J$ is not everywhere parallel to $\dot\gamma$, otherwise it would follow that $\dot\gamma=0$. Comparing condition (b) at $t=0$ and $t=\tfrac{1}{k}$, one obtains that $$J(0)-J(1)=\sum_{i=0}^{k-1} J\left(\tfrac{i}{k}\right)-\sum_{i=0}^{k-1} J\left(\tfrac{1}{k}+\tfrac{i}{k}\right)=0.$$ Moreover, $V(t)=\sum_{i=0}^{k-1} J\left(t+\tfrac{i}{k}\right)$ is the identically null vector field, hence $\D V(t)=0$ for all $t\in [0,1]$. Thus, it follows that $$\D J(0)-\D J(1)=\D V(0)-\D V(\tfrac{1}{k})=0.$$ This concludes the proof, since the same $J$ that degenerates $\gamma$ as a $(g,\p)$--geodesic is also periodic and is not a constant multiple of $\dot\gamma$.
\end{proof}

\begin{remark}\label{tagentisstdegenerate}
Let $\gamma:[0,1]\rightarrow M$ be a periodic $g$--geodesic with period $\tfrac{1}{k}$, $k\ge 2$. Suppose that $\gamma$ admits a nontrivial Jacobi field $J$ such that 
\begin{equation}\label{eq:sommaJ}
\sum\limits_{i=0}^{k} J\left(t+\tfrac{i}{k}\right)=\lambda(t)\dot{\gamma}(t),
\end{equation}
for all $t\in \left[0,\tfrac{1}{k}\right[$. Then, since
the left-hand side in the above equality is a Jacobi field, then $\lambda$ must be an affine function.
However, since $J$ is periodic, then $\lambda$ must be constant, for otherwise the right-hand side of \eqref{eq:sommaJ} would be unbounded as $t\mapsto\pm\infty$. Hence, by adding a suitable multiple of $\dot\gamma$ to $J$, it is possible to obtain a Jacobi field that satisfies conditions (a) and (b) of Definition \ref{defstdeg}. Therefore, in this case $\gamma$ is strongly degenerate.
\end{remark}

\subsection{Main result}\label{bigsec}

We are now ready to prove our main genericity result. More precisely, we prove that the set of semi--Riemannian metrics without degenerate geodesics satisfying an admissible general boundary condition $\p$ is generic in the Whitney $C^k$--topology among metrics for which $\p$ is nondegenerate. It extends the previous genericity statement in \cite{biljavapic} to the general boundary conditions setup described in the last section, which in particular allows one to consider geodesic loops at a point $p$.

Geometrically, \cite{biljavapic} asserts that generically two distinct points are not conjugate. As in Examples \ref{adexgbc} and \ref{exjac}, $\p=P\times\{q\}$ is an admissible general boundary condition, where $P$ is a compact submanifold and $q$ a point. Applied to such $\p$, our result asserts that generically $q$ is not focal to $P$.

The proof is done in two steps. Firstly, we consider the trivially admissible case $\p\cap\Delta=\varnothing$ and apply the abstract genericity criterion using a local perturbation argument. Secondly, we treat the special case $\p\cap\Delta\neq\varnothing$ using its admissibility, since the abstract criterion fails due to the possible presence of strongly degenerate geodesics (see Definition \ref{defstdeg}). Furthermore, we stress that this result is \emph{not} an immediate consequence of the first case $\p\cap\Delta=\varnothing$ and the semi--Riemannian bumpy metric theorem. Indeed, if $\gamma\in\op(M)$ is a periodic $(g,\p)$--geodesic, the notions of degeneracy as a $(g,\p)$--geodesic and as a periodic geodesic \emph{do not} coincide (see Remark \ref{degeneracynotions}). For this, we use a more elaborate argument, which employs both the semi--Riemannian bumpy metric theorem and the abstract genericity criterion in a different way.

\begin{theorem}\label{bigone}
Let $M$ be a smooth $n$--dimensional manifold and $\nu\in\{0,\dots,n\}$ an index. Fix $\mathcal{E}\subset\sect^k_{\mbox{\tiny sym}}(TM^*\otimes TM^*)$ a $C^k$--Whitney type Banach space of tensor fields over $M$ and $\mathcal{A}_\nu\subset\mathcal{E}\cap\met_{\nu}^k(M)$ an open subset. Consider $\p$ an admissible $\nu$--general boundary condition. Then the following is a generic subset in $\mathcal{A}_{\nu,\p}$ $$\mathcal{G}_\p(M)=\Big\{g\in\mathcal{A}_{\nu,\p}: \mbox{ all } (g,\p)\mbox{--geodesics }\gamma\in\op(M) \mbox{ are nondegenerate}\Big\}.$$
\end{theorem}

\begin{proof}
We shall prove the genericity of $\mathcal{G}_\p(M)$ in $\mathcal{A}_{\nu,\p}$ in two steps. Firstly, we suppose that $\p$ satisfies $\p\cap\Delta=\varnothing$, and apply the abstract genericity criterion in Proposition \ref{abstractgenericity}. Secondly, if $\p\cap\Delta\neq\varnothing$, the abstract criterion fails due to possible existence of strongly degenerate geodesics. In this case, we use a more elaborate argument to deal with the periodic geodesic issue, appealing to the semi--Riemannian bumpy metric theorem.

\smallskip
{\bf Case 1.} Suppose $\p\cap\Delta=\varnothing$. From Lemma \ref{fredholmness}, it suffices to prove that condition (ii) of Proposition \ref{abstractgenericity} is satisfied in the geodesic set up described in the beginning of this section. More precisely, we have to prove that given a semi--Riemannian metric $g_0\in\mathcal{A}_{\nu,\p}$ and $\gamma_0$ a $(g_0,\p)$--geodesic with a nontrivial $\p$--Jacobi field $J$ along $\gamma_0$ (see \eqref{pcampodijacobi}), there exists $h\in\mathcal{E}$ such that the right--hand side of \eqref{mixedderivative} does not vanish. For this, we use a \emph{local perturbation argument} along the lines of \cite[Proposition 4.3]{biljavapic}, that employs Lemma \ref{extension}.

Once more we treat two cases separately. First we assume $\gamma_0$ is not a portion of a periodic geodesic with period $\omega <1$. With this assumption, from Proposition \ref{selfintersections}, $\gamma_0$ has only a finite number of self intersections. It is then possible to find an open interval $I\subset [0,1]$ such that $\gamma_0|_I$ is injective, $\gamma_0(I)\cap\gamma_0([0,1]\setminus I)=\varnothing$ and $J$ is not parallel to $\dot{\gamma_0}$ at any time in $I$. Indeed such an interval exists, since the first condition is feasible due to the finiteness of self intersections and the second is also admissible as a consequence of Lemmas \ref{parallelfinite} and \ref{jacobinotparallel}.

In order to find the required $h\in\mathcal{E}$, we now apply Lemma \ref{extension} with $E=TM^*\otimes TM^*$. Let $U\subset M$ be any open subset containing $\gamma_0(I)$ such that $\gamma_0(t)\in U$ if and only if $t\in I$. For instance, $U$ can be taken as the complement of $\gamma_0([0,1]\setminus I)$. Set $H\in\sect(\gamma_0^*E)$ identically null and $K\in\sect(\gamma_0^*E)$ any symmetric bilinear form smooth on $t$, that satisfies $K(\dot{\gamma_0},\dot{\gamma_0})\geq 0$ and $\int_I K_t(\dot{\gamma_0}(t),\dot{\gamma_0}(t))\;\mathrm{d}t>0$. Reducing the size of $I$ if necessary and applying Lemma \ref{extension}, it follows that there exists a globally defined smooth section $h$ of $E$ with compact support contained in $U$ such that $h_{\gamma_0(t)}=0$ and $\nabla_{J_t} h=K_t$ for all $t\in I$. Hence for this $h$, the formula \eqref{mixedderivative} gives$$\int_0^1\left[h(\dot{\gamma_0},\D J)+\tfrac{1}{2}\nabla h(J,\dot{\gamma_0},\dot{\gamma_0})\right]\;\mathrm{d}t=\tfrac{1}{2}\int_I K_t(\dot{\gamma_0}(t),\dot{\gamma_0}(t))\;\mathrm{d}t>0.$$

This concludes the proof of the case $\p\cap\Delta=\varnothing$, when $\gamma_0$ is not a portion of a periodic geodesic of period $\omega <1$.

If the $(g_0,\p)$--geodesic $\gamma_0$ has infinitely many self intersections, one can apply the exact same argument used in the second part of the proof of \cite[Proposition 4.3]{biljavapic} to show that this local perturbation approach above can be adapted. In general terms, the technique consists of a parity argument to find the desired interval $I$ where the local perturbation occurs.

\smallskip
\begin{remark}\label{whenitfails}
More generally, this local perturbation argument used in the previous case of geodesics with finite self intersections can be extended to \emph{any}\footnote{Except the case of a \emph{prime} periodic geodesic $\gamma_0\in\op(M)$, with period $\omega=1$. Recall that a geodesic is said to be prime if it is not obtained as n-fold iteration of some other geodesic.} periodic geodesic that is not strongly degenerate. In fact, if $\gamma_0$ has period $\omega <1$, then $\gamma_0(0)\ne\gamma_0(1)$, and \cite[Proposition 4.3]{biljavapic} applies. If $\gamma_0(0)=\gamma_0(1)$, then $\gamma_0$ is periodic with period $\tfrac{1}{k}$, $k\ge 2$. Suppose $J$ is a nontrivial $\p$--Jacobi field along $\gamma_0$. Then a sufficient condition for the local perturbation to hold is that for some $t_0\in [0,1]$, $$\sum\limits_{i=0}^{k-1} J\left(t_0+\tfrac{i}{k}\right)\ne 0,$$ which holds unless $\gamma_0$ is strongly degenerate. Under this condition, by continuity, it is possible to find an interval $I$ around such $t_0$ with the same properties as the interval $I$ considered above in Case 1. Then, Lemma \ref{extension} guarantees (see Remark \ref{tagentisstdegenerate}) existence of the desired globally defined smooth section $h$, verifying the transversality condition (ii) of Proposition \ref{abstractgenericity}.
\end{remark}

\smallskip
{\bf Case 2.} Assume now $\p\cap\Delta\neq\varnothing$. Recall that $L_\mathrm R$ is the length of curves with respect to the fixed complete Riemannian metric $g_\mathrm R$ on $M$. For each $\alpha\in\N$ define
\begin{equation}\label{eq:Ralpha}
\mathcal{R}_\alpha=\left\{g\in\mathcal{A}_{\nu,\p}:\begin{array}{c}\mbox{ all } (g,\p)\mbox{--geodesics }\gamma\mbox{ with } \\ L_\mathrm R (\gamma)\leq \alpha\mbox{ are nondegenerate} \end{array}\right\}.
\end{equation}

Since $\mathcal{G}_\p(M)=\bigcap_{\alpha\in\N} \mathcal{R}_\alpha$, we shall prove that each $\mathcal R_\alpha$ is open and dense in $\mathcal A_{\nu,\p}$, hence the genericity result will follow from the Baire theorem.

Let us verify the first claim, namely that $\mathcal R_\alpha$ are open. For this, consider a convergent sequence $\{g_\beta\}$ in $\mathcal A_{\nu,\p}\setminus\mathcal R_\alpha$, with $\lim g_\beta=g_\infty$. From definition of $\mathcal R_\alpha$, for each $\beta\in\N$ there exists a degenerate $(g_\beta,\p)$--geodesic $\gamma_\beta$ with $L_\mathrm R(\gamma_\beta)\leq\alpha$. Since $\p$ is compact and $L_\mathrm R(\gamma_\beta)\leq\alpha$, by the Arzel\`a--Ascoli theorem, up to subsequences there exists a convergent sequence $\{t_\beta\}$ in $[0,1]$ with $\lim t_\beta=t_\infty$ such that $\left\|\dot{\gamma_\beta}(t_\beta)\right\|_\mathrm R\leq\alpha$ for all $\beta\in\N$, and $\gamma_\beta(t_\beta)$ converges to $v\in T_{p_\infty}M$, with $p_\infty=\lim \gamma_\beta(t_\infty)$. From continuous dependence of ODE's solutions on initial conditions, it is easy to see that the solution $\gamma_\infty$ of $\D^{g_\infty}\dot\gamma=0$ with initial conditions $\gamma(t_\infty)=p_\infty$ and $\dot\gamma(t_\infty)=v$ is the $C^2$--limit of the sequence of geodesics $\gamma_\beta$. Therefore $\gamma_\infty$ is a $(g_\infty,\p)$--geodesic, and obviously $L_\mathrm R(\gamma_\infty)\leq\alpha$.

Moreover, $\gamma_\infty$ is \emph{nonconstant}. This follows from the fact that $\p$ is an admissible general boundary condition. Hence there exists $a>0$ such that $L_\mathrm R(\gamma_\beta)\geq a$, for large $\beta$, since $g_\beta$ will be in any open neighborhoods of $g_\infty$ in $\mathcal A_{\nu,\p}$.

In order to prove that such $\gamma_\infty$ is a \emph{degenerate} $(g_\infty,\p)$--geodesic, for each $\beta$ let $J_\beta$ be a nontrivial $\p$--Jacobi field along $\gamma_\beta$. Then $J_\beta$ is the solution of a second order linear ODE whose initial conditions converge to initial conditions of the $\p$--Jacobi fields equation along the $g_\infty$--geodesic $\gamma_\infty$. More precisely, for each $\beta$, $J_\beta$ is a nontrivial $\p$--Jacobi that in particular satisfies the $g_\beta$--Jacobi equation $$\D^{g_\beta}J_\beta=R^{g_\beta}J_\beta.$$ By adding a suitable multiple of ${\dot\gamma_\beta}(0)$, one can assume that $J_\beta(0)$ is $g_\mathrm R$--orthogonal to $\dot{\gamma_\beta}(0)$. In addition, using an adequate normalization it is also possible to assume that $\max\{\|J_\beta(0)\|_{\mathrm R},\|\D^{g_\beta} J_\beta(0)\|_{\mathrm R}\}=1.$ Again, up to subsequences, the initial conditions converge, $$\lim J_\beta(0)=v\in T_{\gamma_\infty(0)}M, \; \; \lim \D^{g_\beta} J_\beta(0)=w\in T_{\gamma_\infty(0)}M.$$ By continuity, $v$ is $g_\mathrm R$--orthogonal to $\dot{\gamma_\infty}(0)$, and \begin{equation}\label{maxinfty} \max\{\|v\|_{\mathrm R},\|w\|_{\mathrm R}\}=1. \end{equation} The solution of the $g_\infty$--Jacobi equation along $\gamma_\infty$ with such limit initial conditions is a $\p$--Jacobi field $J_\infty$ that is also the $C^2$--limit of the $\p$--Jacobi fields $J_\beta$. Finally, it is not a trivial $\p$--Jacobi field. Indeed, if $J_\infty$ were a multiple of $\dot{\gamma_\infty}$, since $v$ is $g_\mathrm R$--orthogonal to $\dot{\gamma_\infty}(0)$, it would be $v=0$ and $w=0$, which contradicts \eqref{maxinfty}. Hence $g_\infty\in\mathcal A_{\nu,\p}\setminus\mathcal R_\alpha$, which proves that $\mathcal R_\alpha$ is an open subset.

It still remains to prove the second claim, that $\mathcal R_\alpha$ are dense. For this we define the following subsets of $\mathcal{A}_{\nu,\p}$, $$ \mathcal{B}_\alpha=\left\{g\in\mathcal{A}_{\nu,\p}:\begin{array}{c}\mbox{ all periodic } g\mbox{--geodesics }\gamma\mbox{ with }  \\ L_\mathrm R (\gamma)\leq \alpha\mbox{ are nondegenerate} \end{array}\right\},$$ \smallskip $$ \mathcal{D}_\alpha=\left\{g\in\mathcal{A}_{\nu,\p}: \begin{array}{c} \mbox{ all } g\mbox{--geodesics }\gamma\mbox{ with }L_\mathrm R (\gamma)<\alpha \mbox{ that are} \\ \mbox{periodic or }(g,\p)\mbox{--geodesics are nondegenerate} \end{array} \right\}.$$

It is easy to see that for each $\alpha$, $\mathcal{D}_{\alpha+1}\subset\mathcal{R}_\alpha$. From the semi--Riemannian bumpy metric theorem \cite[Theorem 3.14]{biljavapic2}, each $\mathcal B_\alpha$ is open and dense in $\mathcal A_{\nu,\p}$. Hence to prove that $\mathcal R_\alpha$ is dense in $\mathcal A_{\nu,\p}$, it suffices to prove that $\mathcal D_\alpha$ is dense in $\mathcal B_\alpha$. To this aim, for each $\alpha$ we use the abstract genericity criterion of Proposition \ref{abstractgenericity} again. The setting is as in Case 1, with the only difference being the domain of the generalized energy functional, which we now take as the open subset $$\mathcal U=\mathcal B_\alpha\times\{\gamma\in\op(M):L_\mathrm R(\gamma)<\alpha\}.$$ This means that we are dealing only with \emph{bumpy} metrics, that is, metrics without degenerate periodic geodesics.

Let us prove that the hypothesis of Proposition \ref{abstractgenericity} are verified also in this context, concluding the proof. Condition (i) follows again from Lemma \ref{fredholmness}. As for \emph{transversality} condition (ii), it would only fail in the presence of strongly degenerate geodesics (see Remark \ref{whenitfails}). Indeed, for all non strongly degenerate geodesics $\gamma_0$ and $\p$--Jacobi fields $J$ along $\gamma_0$, the sufficient condition mentioned in Remark \ref{whenitfails} is verified, hence Lemma \ref{extension} can be applied to some interval $I$ with the same properties as in Case 1. This is true even if $\gamma_0$ is a prime periodic geodesic, since it cannot be degenerate once the considered domain is a set of bumpy metrics $\mathcal B_\alpha$.

However, if $\gamma_0$ is a strongly degenerate $(g_0,\p)$--geodesic, then it admits a Jacobi field $J$ which satisfies (b) of Definition \ref{defstdeg}. For this $J$, the right--hand side of \eqref{mixedderivative} is identically null for any section $h$ of $TM^*\otimes TM^*$. In this strongly degenerate case, condition (ii) would not hold. Nevertheless, there cannot be critical points of the form $(g_0,\gamma_0)$, where $\gamma_0$ is a strongly degenerate $(g_0,\p)$--geodesic. This follows from Proposition \ref{stronglydegenerate}, since $\gamma_0$ would also be a degenerate periodic geodesic, contradicting $g_0\in\mathcal B_\alpha$. Thus condition (ii) is verified and the abstract genericity criterion applies. Therefore $\mathcal R_\alpha$ is generic, in particular dense, in $\mathcal A_{\nu,\p}$ for each $\alpha$.

This concludes the proof that $\mathcal G_\p(M)$ is generic in $\mathcal A_{\nu,\p}$.
\end{proof}

\subsection{\texorpdfstring{Genericity in the $C^\infty$--topology}{Genericity in the smooth case}}\label{sec:cinfty}

We conclude this section showing how to extend the notion of genericity given above to the space of metrics endowed with the weak $C^\infty$--topology. Since this topology cannot be induced by a Banach space structure on the set $\sect^k_{\mbox{\tiny sym}}(TM^*\otimes TM^*)$ of symmetric tensors on $M$, Proposition \ref{abstractgenericity} cannot be applied. The appropriate argument uses ideas from \cite{fhs} and although it is basically contained in previous works \cite{biljavapic,gj}, will be repeated here for reader's convenience.

Let us rename $\mathcal A_\nu,\, \mathcal A_{\nu,\p}$ and $\mathcal G_\p(M)$ as $\mathcal A^k_\nu,\,\mathcal A^k_{\nu,\mathcal P}$ and $\mathcal G^k_\p(M)$ to stress dependence on $C^k$ regularity of tensor fields. Let us set $\mathcal A^\infty_{\nu,\p}=\bigcap_{k\in\N}\mathcal A^k_{\nu,\p}$ and analogously $$\mathcal G^\infty_\p(M)=\bigcap_{k\in\N}\mathcal G^k_\p(M).$$ The main result stated above, Theorem \ref{bigone}, asserts that $\mathcal G^k_\p(M)$ is generic in $\mathcal A^k_{\nu,\p}$, for all $k$, and we claim it also holds in the $C^\infty$--topology.

\begin{proposition}
$\mathcal G^\infty_\p(M)$ is generic in $\mathcal A^\infty_{\nu,\p}$.
\end{proposition}

\begin{proof}
Rename as $\mathcal R^k_\alpha$ the set \eqref{eq:Ralpha} of metrics in $\mathcal A_{\nu,\p}^k$ defined in the proof of Theorem \ref{bigone}, such that all $(g,\p)$--geodesics with $g_\mathrm R$--length less or equal to $\alpha$ are nondegenerate. Let us also set $\mathcal R^\infty_\alpha=\bigcap_{k\in\N} \mathcal R^k_\alpha$. From the same argument used in the proof, it follows that $\mathcal R^k_\alpha$ is open in $\mathcal A^k_{\nu,\p}$ for $k=2,\ldots,+\infty$. Therefore, it is only left to prove that $\mathcal R^\infty_\alpha$ is dense in $\mathcal A^\infty_{\nu,\p}$, for all positive integers $\alpha$. This implies that $\mathcal G^\infty_\p(M)=\bigcap_{\alpha\in\N}\mathcal R^\infty_\alpha$ will be a countable intersection of open dense subsets of $\mathcal A^\infty_{\nu,\p}$, hence it is generic in $\mathcal A^\infty_{\nu,\p}$.

In order to verify that $\mathcal R^\infty_\alpha$ is dense in $\mathcal A^\infty_{\nu,\p}$, for each $\alpha$ we argue as follows. Note that $\mathcal R^k_\alpha$ contains $\mathcal G^k_\p(M)$, which is generic (Theorem \ref{bigone}) in $\mathcal A^k_{\nu,\p}$. Therefore $\mathcal R^k_\alpha$ is dense in $\mathcal A^k_{\nu,\p}$ (and open, as already mentioned). Moreover $\mathcal A^\infty_{\nu,\p}$ is dense in $\mathcal A^k_{\nu,\p}$ for each $k\in\N,\,k\ge 2$. Observe that $\mathcal R^\infty_\alpha=\mathcal A^\infty_{\nu,\p}\cap\mathcal R^k_\alpha$, and is therefore dense in $\mathcal A^k_{\nu,\p}$. In fact, it is the intersection of a dense subset with and open and dense subset of $\mathcal A^k_{\nu,\p}$. Thus $\mathcal R^\infty_\alpha$ is dense in the intersection $\cap_k A^k_{\nu,\p}=\mathcal A^\infty_{\nu,\p}$, and this concludes the proof.
\end{proof}


\begin{thebibliography}{200}
\bibitem{AbbMej2}{{\sc A. Abbondandolo \and P. Majer}, \emph{A Morse complex for infinite dimensional manifolds. I}, Adv.\ Math.\ \textbf{197} 2 (2005) 321--410.}
\bibitem{AbbMej}{{\sc A. Abbondandolo \and P. Majer}, \emph{A Morse complex for Lorentzian geodesics}, Asian J. Math. {\bf 12} 3 (2008), 299--319.}
\bibitem{abraham}{{\sc R. Abraham}, {\it Bumpy metrics}, in Global Analysis (Proc. Sympos. Pure Math., Vol XIV, Berkeley, Calif., 1968), AMS, Providence, R.I., 1970, 1--3.}
\bibitem{anosov}{{\sc D. V. Anosov}, {\it Generic properties of closed geodesics}, Izv. Akad. Nauk SSSR Ser. Mat. {\bf 46} (1982), no 4, 675--709.}
\bibitem{biljavapic}{{\sc L. Biliotti, M. A. Javaloyes, \and P. Piccione}, {\it Genericity of nondegenerate critical points and Morse geodesic functionals}, Indiana Univ. Math. J. (2009) \textbf{58} 4, 1797--1830.}
\bibitem{biljavapic2}{{\sc L. Biliotti, M. A. Javaloyes, \and P. Piccione}, {\it On the semi-Riemannian bumpy metric theorem}, arXiv:0907.4022v1, preprint 2009.}
\bibitem{chillingsworth}{{\sc D. Chillingsworth}, \emph{A global genericity theorem for bifurcations in variational problems}, J. Func. Anal. {\bf 35} (1980), 251-278.}
\bibitem{fhs}{{\sc A. Floer, H. Hofer, \and D. Salamon}, {\it Transversality in elliptic Morse theory for the symplectic action},  Duke Math. J. \textbf{80} 1 (1995), 251--292.}
\bibitem{ggp}{{\sc R. Giamb\`o, F. Giannoni, \and P. Piccione}, {\it Genericity of nondegeneracy for light rays in stationary spacetimes}, Commun. Math. Phys., \textbf{287} 3, 903--923 (2009).}
\bibitem{gj}{{\sc R. Giamb\`o \and M. A. Javaloyes}, {\it Addendum to ``Genericity of nondegeneracy for light rays in stationary spacetimes''}, Commun. Math. Phys., to appear (2009).}
\bibitem{kn}{{\sc S. Kobayashi \and K. Nomizu}, {\it Foundations of differential geometry}, vol II, John Wiley and Sons, 1969.}
\bibitem{lang}{{\sc S. Lang}, {\it Fundamentals of Differential Geometry}, Springer Verlag (Graduate Texts in Mathematics), 1999.}
\bibitem{mp}{{\sc K. R. Meyer, \and J. Palmore}, {\it A generic phenomenon in conservative Hamiltonian systems}, in Global Analysis (Proc. Sympos. Pure Math., Vol XIV, Berkeley, Calif., 1968), AMS, Providence, R.I., 1970, 861--866.}
\bibitem{palais}{{\sc R. Palais}, {\it Foundations of Global Nonlinear Analysis}, W. A. Benjamin, 1968.}
\bibitem{PicTauJMP} {{\sc P. Piccione, D. V. Tausk}, {\it A note on the Morse index theorem for geodesics between submanifolds in semi-Riemannian geometry}, J. Math.\ Phys.\  \textbf{40} (1999), no.\ 12, 6682--6688.}
\bibitem{smale}{{\sc S. Smale}, {\it An infinite dimensional version of Sard's theorem}, Amer.\ J. Math.\ {\bf 87} (1965), 861-866.}
\bibitem{white}{{\sc B. White}, {\it The space of minimal submanifolds for varying Riemannian metrics}, Indiana Univ. Math. J. {\bf 40} (1991), 161-200.}
\end{thebibliography}
\end{document}